\documentclass[11pt,a4paper]{article} 

\usepackage[utf8]{inputenc}
\usepackage[T1]{fontenc}
\usepackage{amsmath}
\usepackage{amsfonts}
\usepackage{ae}
\usepackage{units}
\usepackage{icomma}
\usepackage{color}
\usepackage{graphicx}
\usepackage{bbm}
\usepackage{amsthm}
\usepackage{amssymb}
\usepackage{cancel}
\usepackage{fullpage}
\usepackage{type1cm}
\usepackage{eso-pic}
\usepackage{hyperref}
\usepackage{xspace}
\usepackage{enumitem}
\usepackage{bbm}
\usepackage{verbatim}
\usepackage{subfig}
\usepackage{tikz}
\usepackage{pgfplots}
\usetikzlibrary{arrows}
\usepackage{amsrefs}
\usepackage{enumitem}

\newcommand{\abs}[1]{\ensuremath{\left| #1 \right|}}

\newcommand{\E}{\mathbb{E}}
\newcommand{\Trip}{\ensuremath{T}}
\newcommand{\trip}{\ensuremath{t}}
\renewcommand{\Pr}{\mathbb{P}}

\renewcommand{\epsilon}{\varepsilon}

\newtheorem{theorem}{Theorem}[section]
\newtheorem{lemma}[theorem]{Lemma}
\newtheorem{proposition}[theorem]{Proposition}

\newtheorem{claim}[theorem]{Claim}

\newtheorem{conjecture}[theorem]{Conjecture}
\theoremstyle{definition}

\newtheorem{remark}[theorem]{Remark}

\numberwithin{equation}{section}
\numberwithin{theorem}{section}

\begin{document}
\title{Most edge-orderings of $K_n$ have maximal altitude
}
\author{Anders Martinsson\\~\\ \small Institute of Theoretical Computer Science\\ \small ETH Zürich\\ \small 8092 Zürich, Switzerland\\
\small Email: \href{mailto:maanders@inf.ethz.ch}{maanders@inf.ethz.ch}}


\maketitle

\begin{abstract}
Suppose the edges of the complete graph on $n$ vertices are assigned a uniformly chosen random ordering. Let $X$ denote the corresponding number of Hamiltonian paths that are increasing in this ordering. It was shown in a recent paper by Lavrov and Loh that this quantity is non-zero with probability at least $1/e-o(1)$, and conjectured that $X$ is asymptotically almost surely non-zero. In this paper, we prove their conjecture. We further prove a partial result regarding the limiting behaviour of $X$, suggesting that $X/n$ is log-normal in the limit as $n\rightarrow\infty$. A key idea of our proof is to show a certain relation between $X$ and its size-biased distribution. This relies heavily on estimates for the third moment of $X$.
\end{abstract}

\section{Introduction}
The \emph{altitude of an edge-ordered graph} $(G, \preceq)$ is the length of the longest monotone (self-avoiding) path in $G$. The \emph{altitude of a graph} $G$ is the smallest altitude of any edge-ordered version of $G$. We denote the altitude of $G$ by $f(G)$.

This quantity was first proposed in 1971 by Chv\'atal and Koml\'os \cite{CK71} who asked for the altitude of the complete graph on $n$ vertices, $K_n$. In their paper, they relayed personal communications from R. Graham and D. Kleitman that $\Omega(\sqrt{n}) = f(K_n) \leq \left(\frac{3}{4}+o(1)\right) n$. Two years later, Graham and Kleitman published their result \cite{GK73}, showing that $$\sqrt{n-3/4}-1/2\leq f(K_n) \leq 3n/4+O(1),$$ and conjectured that $f(K_n)$ is ``closer to the upper bound''. The constant $3/4$ in the upper bound has since been improved to $\frac{1}{2}+o(1)$ by Calderbank, Chung, and Sturtevant in 1984 \cite{CCS84}, which is currently the best known upper bound. This result is based on an enticingly simple construction of an edge-ordering: Label the vertices of $K_n$ by unique binary strings of length $\lceil \log_2 n \rceil$, and order the edges lexicographically according to the bitwise XOR of its end-points, with ties broken arbitrarily.

In his Master's thesis from 1973 \cite{R73}, Rödl considered the altitude of graphs with given average degree. Generalizing the lower bound for the complete graph, he showed that if $G$ has average degree $d$, then $f(G)\geq (1-o(1))\sqrt{d}$. For sufficiently dense graphs Rödl's result was recently improved by Milans \cite{M15+} who proved that, for $d \gg n^{2/3} \left(\ln n\right)^{4/3}$, we have $$f(G)= \Omega\left({d} {n^{-1/3}\left(\ln n\right)^{-2/3}}\right).$$ In particular, $$f(K_n)\geq \left( \frac{1}{20}+o(1)\right)\left(\frac{n}{\ln n}\right)^{2/3},$$ finally making significant improvement on Graham's and Kleitman's lower bound after over four decades.

While any progress on the altitude of the complete graph has been scarce for the last three decades, a number of studies have appeared during this time which consider the altitude of other classes of graphs. Roditty, Shoham, and Yuster \cite{RSY01} asked for the highest possible altitude of a planar graph. They showed that by taking a large cycle and adding two vertices connected to everything on the cycle one obtains a planar graph of altitude at least $5$. On the other hand, by proving that the edges of any planar graph $G$ can be partitioned into three forests, they show that $f(G)\leq 9$. Using a result by Gon\c{c}alves \cite{G07} that any planar graph can be partitioned into four so-called forests of caterpillars, this upper bound can be improved to $8$.

It can be seen as a consequence of Vizing's theorem (see e.g. Theorem 1.13 in \cite{B78}) that the altitude of any graph $G$ is at most $\Delta(G)+1$ where $\Delta(G)$ denotes the maximum degree of any vertex in $G$. Given this bound, it is natural to ask how large one can make $f(G)$ as a function of $\Delta(G)$. It was observed by Alon \cite{A03} that, for any $k\geq 2$, any $k$-regular graph $G$ with girth strictly greater than $k$ (these are well-known to exist) has altitude at least $k=\Delta(G)$. Hence for each $\Delta\geq 2$, the largest attainable altitude is either $\Delta$ or $\Delta+1$. It is not too hard to see that any odd cycle of length at least $5$ has altitude $3$. It was shown by Mynhardt, Burger, Clark, Falvai and Henderson \cite{MBCFH05} that there exist cubic ($3$-regular) graphs, for instance the so-called flower snarks, with altitude $4$. Hence for $\Delta=2$ and $3$ the upper bound of $\Delta+1$ can be attained exactly, but the question remains open for $\Delta \geq 4$ whether the maximally attainable altitude is $\Delta$ or $\Delta+1$.

De Silva, Molla, Pfender, Retter and Tait \cite{SMPRT16} prove that the altitude of the $d$-dimensional hypercube lies between $\frac{d}{\ln d}$ and $d$. In the same paper, they show that for any $\omega(n)\rightarrow\infty$ we have $$f(G(n, p)) \geq (1-o(1)) \min\left( \sqrt{n}, \frac{np}{\omega(n) \ln n}\right)$$ with high probability as $n\rightarrow\infty$, where $G(n, p)$ is the Erd\H{o}s-R\'enyi random graph. Here, a sequence of events $\{E_n\}_{n=1}^\infty$ is said to hold \emph{with high probability} if $\mathbb{P}(E_n)\rightarrow 1$ as $n\rightarrow\infty$. Moreover, as the altitude is at most the maximum degree plus one, this bound is sharp up to a logarithmic factor for $p\leq \frac{1}{\sqrt{n}}$.

Katreni\v{c} and Semani\v{s}in \cite{KS10} proved that the problem of deciding whether or not a given edge-ordered graph contains a monotone Hamiltonian path is NP-complete. Hence, the (decision version of the) problem of computing the altitude of an edge-ordered graph is NP-hard. It remains an open question whether the problem of computing the altitude of a non-edge-ordered graph is also NP-hard.

In a recent article \cite{LL16}, Lavrov and Loh considered the altitude of a uniformly chosen edge-ordering of $K_n$. Recall that a (self-avoiding) path in a graph $G$ is said to be Hamiltonian if it visits all vertices. Their main result states that, with high probability, the altitude of such an edge-ordering is at least $0.85n$, and with probability at least $\frac{1}{e}-o(1)$ the edge-ordering contains a monotone Hamiltonian path, that is, the altitude is $n-1$. They consequently made the natural conjecture that monotone Hamiltonian paths should exist in this graph with high probability. The aim of this paper is to prove this conjecture.


Consider a uniformly chosen random edge-ordering of $K_n$. Let $X=X_n$ denote the number of Hamiltonian paths that are increasing with respect to this ordering. As there are $n!$ Hamiltonian paths in $K_n$, and each path is increasing with probability $1$ in $(n-1)!$, we have $\E X = n$. To prove a lower bound on $\Pr(X>0)$, Lavrov and Loh gave an elegant argument for estimating the second moment of $X$, yielding $\E X^2 \sim en^2$, where $f_n\sim g_n$ denotes that $\frac{f_n}{g_n}\rightarrow 1$ as $n\rightarrow\infty$. Note that this means that the standard deviation of $X$ is $\sim\sqrt{e-1}$ times its expectation, which is too large for Chebyshev's inequality to imply concentration. On the other hand, using the Paley-Zygmund lemma, Lavrov and Loh's estimates imply
$$\Pr(X>0) \geq \frac{\left(\E X\right)^2}{\E X^2} = \frac{1}{e}-o(1),$$
as stated above.

The key idea to our approach is relate the distribution of $X$ to its so-called size-biased distribution. As we shall see, this boils down to showing certain third moment estimates of $X$. We have the following result.

\begin{theorem}\label{thm:main} A uniformly chosen edge-ordering of $K_n$ contains an increasing Hamiltonian path with high probability as $n\rightarrow\infty$.
\end{theorem}

In fact, out approach tells us much more about the limiting behaviour of $X$ as $n\rightarrow\infty$, which is summarized in the technical result below. Before delving into that, let us first mention some corollaries. First, for any $M > 0$ and any $k=1, 2, 3, \dots$, we have the tail bounds
\begin{equation}\label{eq:lowertail}
\limsup_{n\rightarrow\infty} \Pr\left( \frac{X}{n} \leq \frac{1}{M}\right) \leq e^{k(k+1)/2}\frac{1}{M^k},
\end{equation}
and
\begin{equation}\label{eq:uppertail}
\limsup_{n\rightarrow\infty} \Pr\left( \frac{X}{n} > M\right) \leq e^{k(k-1)/2}\frac{1}{M^k}.
\end{equation}
Hence, with high probability, $X$ is of order $n$. On the other hand, for any $M>0$
\begin{equation}\label{eq:maybesmall}
\liminf_{n\rightarrow\infty} \Pr\left( \frac{X}{n} \leq \frac{1}{M}\right) > 0,
\end{equation}
and
\begin{equation}\label{eq:maybelarge}
\liminf_{n\rightarrow\infty} \Pr\left( \frac{X}{n} > M\right) > 0,
\end{equation}
so $X/n$ will does not have compact support as $n\rightarrow\infty$. In particular, $X/n$ is not concentrated on a single value. Furthermore, it is possible to define a truncation $\hat{X}$ of $X$ such that $\Pr(\hat{X}=X)\rightarrow 1$ as $n\rightarrow\infty$ and for any $k=1, 2, \dots$,
\begin{equation}\label{eq:truncmoment}
\E \hat{X}^k \sim e^{k(k-1)/2} n^k.
\end{equation}
As mentioned above, it is known that $\E X = n$ and $\E X^2 \sim e n^2$, and we will show later in this paper that $\E X^3 \sim e^3 n^3$, which all agree with this formula. In fact, under the assumption that there exist a constant $C_k>0$ such that $\E X^k \leq C_k n^k$ for each $k=1, 2, \dots$, it follows that \eqref{eq:truncmoment} holds for $\hat{X}=X$.

Considering the aforementioned properties of $X$, it is reasonable to expect that $\{X_n/n\}_{n=1}^\infty$ should converge to some non-trivial distribution with mean $1$ and standard deviation $\sqrt{e-1}$ as $n\rightarrow\infty$. While this remains an open problem, our second main result gives a partial result to this end, which strongly suggests that the limit exists and is \emph{log-normal}. The log-normal distribution, $\log \mathcal{N}(\mu, \sigma)$, with parameters $\mu$ and $\sigma$ is the distribution of $Y=e^Z$ where $Z$ has normal distribution with mean $\mu$ and standard deviation $\sigma$.

Since $\E X/n = 1$ for all $n$, the sequence $\{X_n/n\}_{n=1}^\infty$ is tight, meaning that no mass of this sequence escapes to infinity as $n\rightarrow\infty$. By Prokhorov's theorem, this implies that for any sequence $\{X_{m_i}/m_i\}_{i=1}^\infty$ with $m_i\rightarrow\infty$ there is a subsequence $\{X_{n_i}/n_i\}_{i=1}^\infty$ that converges in distribution, that is, there exists an increasing and right-continuous function $F:[0, \infty)\rightarrow[0, 1]$ such that $$\Pr(X_{n_i}/n_i\leq x) \rightarrow F(x)\text{ as }i\rightarrow\infty$$ for any $x\geq 0$ where $F$ is continuous. Moreover, for any $x \geq 0$, $\limsup_{i\rightarrow\infty} \Pr(X_{n_i}/n_i\leq x) \leq F(x)$.


\begin{theorem}\label{thm:limitpointprops}
Let $F:[0,\infty)\rightarrow [0, 1]$ denote the cumulative distribution function of the limit of any weakly converging subsequence $\{X_{n_i}/n_i\}_{i=1}^\infty$ as above. We then have
\begin{equation}\label{eq:lognormalmoments}
\int_0^\infty x^k\,dF(x) = e^{k(k-1)/2}
\end{equation}
for any (not necessarily positive) integer $k$. That is, $F$ has the same moments as a log-normal random variable with $\mu=-\frac{1}{2}$ and $\sigma=1$. Moreover, if we let $G(t)=F(e^t)$, equivalently we let $G(t)$ be the CDF of the limit of $\ln\left(X_{n_i}/n_i\right)$, then $G(t)$ can be written as
\begin{equation}\label{eq:periodicmeasure}
G(t) = \int_{-\infty}^t e^{ -(s+\frac{1}{2})^2/ 2} d\nu(s)
\end{equation}
where $\nu(t)$ is a $1$-periodic positive measure on $\mathbb{R}$.
\end{theorem}
We here say that a measure $\mu$ is $1$-periodic if $\mu$ is invariant under integer translations, that is, $\mu(A) = \mu(A+k)$ for any $k\in\mathbb{Z}$ and any measurable set $A$.

An important caveat relating to this theorem is that the log-normal distribution is \emph{M-indeterminate}, meaning that there exist other random variables that have the same moments. Hence \eqref{eq:lognormalmoments} does not in fact prove that $F$ is log-normal. To see this, one can show (this will be proven as part of Theorem \ref{thm:limitpointprops} below) that \eqref{eq:periodicmeasure} implies \eqref{eq:lognormalmoments} in the sense that for any $1$-periodic positive measure $\nu$ such that $\int_{-\infty}^\infty e^{ -(s+\frac{1}{2})^2/ 2} d\nu(s) = 1$, the corresponding distribution $F(x) = \int_{-\infty}^{\ln x} e^{ -(t+\frac{1}{2})^2/ 2} d\nu(t)$ has log-normal moments. In \cite{ST05} further examples of distributions with these moments are constructed, which implies that the characterization of $F$ in \eqref{eq:periodicmeasure} is stronger than just characterizing the moments.

\begin{proof}[Proof of properties \eqref{eq:lowertail}---\eqref{eq:truncmoment}] Assuming one of \eqref{eq:lowertail}--\eqref{eq:maybelarge} does not hold, we can construct a weakly convergent subsequence $\{X_{n_i}/n_i\}_{i=1}^\infty$ such that the analogous statement for this subsequence also is false. Denote the limit distribution by $F$. But  by  \eqref{eq:lognormalmoments} we have that $$F(1/M) \leq \int_0^\infty \frac{1}{(Mx)^k} \, dF(x) = e^{k(k+1)/2}/M^k$$ and $$1-F(M) \leq \int_0^\infty \left(\frac{x}{M}\right)^k\,dF(x) = e^{k(k-1)/2}/M^k,$$
and by \eqref{eq:periodicmeasure} we must have $$F(1/M) = \int_{\left(-\infty, -\ln M\right]}
e^{ -(t+\frac{1}{2})^2/ 2} \, d\nu(t) > 0$$
and $$1-F(M) = \int_{\left(\ln M, \infty\right)}
 e^{ -(t+\frac{1}{2})^2/ 2} \, d\nu(t) > 0,$$ as by $1$-periodicity, $\nu( [s, s+1) )>0$ for any $s\in \mathbb{R}$. For each of \eqref{eq:lowertail}--\eqref{eq:maybelarge}, this gives us a contradiction.

For \eqref{eq:truncmoment}, we define $X_{M,n}=\min(X_n, M\cdot n)$. For each $n\geq 1$, let $M_n$ be the largest integer $0 \leq M \leq n$ such that $$\abs{\E \left[(X_{M, n}/n)^k\right] - e^{k(k-1)/2}}\leq \frac{3 e^{k(k+1)/2}}{M}$$ for all $1\leq k \leq M$ (this condition holds trivially for $M=0$). Note that if $M_n\rightarrow\infty$ as $n\rightarrow\infty$ then \eqref{eq:truncmoment} is satisfied for $\hat{X}_n = X_{M_n, n}$, hence it suffices to show that $M_n$ diverges.

If $M_n$ does not diverge, there must be constants $k\leq M$ such that $\abs{\E \left[(X_{M, n}/n)^k\right] - e^{k(k-1)/2}} > \frac{3 e^{k(k+1)/2}}{M}$ for infinitely many $n$. Constructing a weakly converging sequence $\{X_{n_i}/n_i\}_{i=1}^\infty$ among these, we get as $i\rightarrow\infty$ that
\begin{align*}
\frac{3 e^{k(k+1)/2}}{M} &< \abs{ \E\left[ \left( X_{M,n_i}/n_i\right)^k\right] - e^{k(k-1)/2}}\\ &\rightarrow \abs{ \int_0^M x^k\,dF(x) + M^k \left(1-F(M)\right) - \int_0^\infty x^k\,dF(x)}\\ &\leq \int_M^\infty x^k\,dF(x)+M^k (1-F(M))\\ &\leq \frac{1}{M} \int_0^\infty x^{k+1}\,dF(x) + M^k e^{k(k+1)/2}\frac{1}{M^{k+1}}\\
&= \frac{2 e^{k(k+1)/2}}{M},
\end{align*}
which is a contradiction. Hence $M_n$ diverges, as desired.

Finally, if we, in addition, assume that there exist bounds $\E X^k \leq C_k n^k$ for each $k=1, 2, \dots$, we can write $$\abs{\E \left[(X/n)^k\right] - \E \left[(X_{M, n}/n)^k\right]} \leq \E \left[\mathbbm{1}_{X/n > M_n} (X/n)^k\right] \leq \frac{1}{M_n}\E \left[(X/n)^{k+1}\right] \leq \frac{C_{k+1}}{M_n}\rightarrow 0$$ as $n\rightarrow\infty$. Hence for each $k=1, 2, \dots$, $\E \left[(X/n)^k\right]$ has the same limit as $\E \left[(\hat{X}/n)^k\right]$.



\end{proof}



The question remains open whether or not $\{X_n/n\}_{n=1}^\infty$ has a limiting distribution, and in that case which of the distributions of the form prescribed in Theorem \ref{thm:limitpointprops} it is. It seems that new ideas are needed to make any further progress on this problem. Nevertheless, I believe that Theorem \ref{thm:limitpointprops} provides strong evidence for the following statement.
\begin{conjecture}\label{conj:lognormal}
As $n\rightarrow\infty$, $\frac{X}{n}$ converges in distribution to a $\log \mathcal{N}(-\frac{1}{2}, 1)$ random variable.
\end{conjecture}

In the remaining parts of the article, we will prove Theorems \ref{thm:main} and \ref{thm:limitpointprops}. We remark that Theorem \ref{thm:main} can be seen as a consequence of Theorem \ref{thm:limitpointprops} and in particular \eqref{eq:lowertail}, but we will still give a direct proof of this statement to illustrate our proof approach. Section \ref{sec:XYrelation} gives the main idea of our approach and shows how our results can be reduced to showing third moment estimates for $X$. These estimates will then be derived in Section \ref{sec:proofofEX3est}, completing the proof of both statements.

\section{\label{sec:XYrelation}Proof of Theorems \ref{thm:main} and \ref{thm:limitpointprops}}

For any non-negative integer-values random variable $\xi$ with $0<\E \xi < \infty$, we say that a random variable $\xi^s$ has the $\xi$ size-biased distribution if for all $k=0, 1, 2, \dots$
$$\mathbb{P}(\xi^s = k) = \frac{ k \mathbb{P}(\xi=k)}{\mathbb{E}\xi}.$$
Note that $\sum_{k=0}^\infty \mathbb{P}(\xi^s = k) = \frac{1}{\mathbb{E}\xi}\sum_{k=0}^\infty k\mathbb{P}(\xi=k) = \frac{1}{\mathbb{E}\xi} \mathbb{E}\xi = 1$, so this indeed defines a probability measure for $\xi^s$.

The main idea in our proof is to show how the distribution of $X=X_n$ relates to its size-biased distribution. Let $P_0$ be a fixed Hamiltonian path in $K_n$. We let $Y=Y_n$ be the number of increasing Hamiltonian paths in $K_n$ one would get if one were to take the uniformly chosen edge-ordering and switch positions of edges along $P_0$ in the ordering such that $P_0$ becomes increasing. One can observe that the modified edge-ordering has the same distribution as a uniformly chosen edge-ordering conditioned on $P_0$ being increasing.


\begin{proposition}\label{prop:ezrel} The random variable $Y$ has the $X$ size-biased distribution. That is, for any $k\geq 0,$
\begin{equation*}
\Pr \left(Y=k\right) = \frac{k}{n} \Pr(X=k).
\end{equation*}
\end{proposition}
\begin{proof}
Note that the conditional probability that $P_0$ is increasing given $X=k$ is, by symmetry, $k/n!$. Hence, by Bayes' theorem,
$$\Pr(Y=k) = \Pr(X=k \vert P_0\text{ inc}) = \frac{\Pr(P_0\text{ inc}\vert X=k) \Pr(X=k)}{\Pr(P_0\text{ inc})} = \frac{ (k/n!) \Pr(X=k)}{1/(n-1)!},$$
which simplifies to the expression above.
\end{proof}

Our main result is a consequence of the surprisingly simple property that, for large $n$, $Y$ is approximately $X$ scaled up by a constant factor, in this case $e$. In other words, for large $n$, conditioning on one fixed Hamiltonian path being increasing will multiply the number of increasing such paths by $e$, but otherwise not affect the distribution. This is captured in the following theorem.

\begin{theorem}\label{thm:hardrel}
We have $\E\left[\left(Y-e\,X\right)^2\right]= o(n^2).$
\end{theorem}

The proof of this theorem involves some rather involved combinatorial arguments and will take up the majority of the article. Therefore, before discussing how to prove it, let us first see why it implies our main results.

\begin{proof}[Proof of Theorem \ref{thm:main}.]
Let $x, \varepsilon>0$. Observe that if $X \leq x n$, then trivially either $Y\leq (e+\varepsilon)xn$ or $Y-e\,X > \varepsilon x n$, where the probability for the latter case tends to $0$ as $n\rightarrow\infty$ by Theorem \ref{thm:hardrel}. Hence, we have
\begin{equation*}
\limsup_{n\rightarrow\infty} \mathbb{P}\left( X_n \leq x n\right) \leq \limsup_{n\rightarrow\infty} \mathbb{P}\left( Y_n \leq (e+\varepsilon) x n\right).
\end{equation*}
By Proposition \ref{prop:ezrel}, we further have
\begin{equation*}
\Pr\left( Y_n \leq (e+\varepsilon)x n\right) = \sum_{k=0}^{\lfloor (e+\varepsilon)x n\rfloor} \frac{k}{n} \Pr\left(X_n=k\right) \leq (e+\varepsilon)x \sum_{k=0}^{\infty} \Pr\left(X_n=k\right) = (e+\varepsilon)x.
\end{equation*}
Hence $\limsup_{n\rightarrow\infty} \Pr(X_n=0) \leq \limsup_{n\rightarrow\infty} \Pr(X_n \leq x n) \leq (e+\varepsilon)x.$ The theorem follows by letting $x\rightarrow 0$.
\end{proof}

\begin{proof}[Proof of Theorem \ref{thm:limitpointprops}.] Given a sequence $X_{n_i}/n_i$ that converges in distribution to a random variable $\mathcal{X}$ with cumulative distribution function $F$, consider the corresponding sequence $Y_{n_i}/n_i$. By Proposition \ref{prop:ezrel} one can show that
\begin{equation}\label{eq:ezrellim}
\Pr( Y_{n_i}/n_i  \leq x ) \rightarrow \int_0^x y \,dF(y)\text{ as } i\rightarrow\infty
\end{equation}
for all points $x\geq 0$ where $F$ is continuous. To see this formally, define $g_x(y):=y\cdot \mathbbm{1}_{y\in[0, x]}$ and note that, by Proposition \ref{prop:ezrel}, $\Pr( Y_{n_i}/n_i  \leq x ) = \E g_x\left( X_{n_i}/n_i \right)$. By the Mapping Theorem, $g_x(X_{n_i}/n_i)$ converges in distribution to $g_x(\mathcal{X})$ for any point of continuity $x\geq 0$ of $F$. As $g_x(X_{n_i}/n_i)$ is bounded in absolute value by $x$, it follows that $\E g_x(X_{n_i}/n_i) \rightarrow \E g_x(\mathcal{X}) = \int_0^\infty g_x(y)\,dF(y) = \int_0^x y\,dF(y)$ as $i\rightarrow\infty$, as desired. Alternatively, this can be shown in a more direct fashon, see Theorem 2.3 and the preceding discussion in \cite{AGK15+}, by noting that the sequence $X_{n_i}/n_i$ is uniformly integrable, which follows from the fact that the sequence $X_n/n$ has bounded second moment.

Furthermore, by Theorem \ref{thm:hardrel} and the fact that $X_{n_i}/n_i$ converges to $\mathcal{X}$ in distribution as $i\rightarrow\infty$, the sequence $Y_{n_i}/n_i$ converges to $e\mathcal{X}$ in distribution. Thus
\begin{equation}\label{eq:hardrellim}
\Pr( Y_{n_i}/n_i  \leq x ) \rightarrow F(e^{-1} x) \text{ as } i\rightarrow\infty
\end{equation}
for all points $x\geq 0$ such that $F$ is continuous at $e^{-1}x$.

As a consequence of \eqref{eq:ezrellim} and \eqref{eq:hardrellim}, we have
\begin{equation*}
F(e^{-1}x) = \int_0^x y\,dF(y),
\end{equation*}
or equivalently, letting $G(t) = F(e^t)$, that is $G$ is the cumulative distribution function of $\ln\mathcal{X}$, we get
\begin{equation*}
G(t-1) = \int_{-\infty}^t\,dG(s-1) = \int_{-\infty}^t e^{s}\,dG(s).
\end{equation*}
Here $dG(s-1)$ denotes that the first integral is taken over the probability measure of a random variable with the cumulative distribution function $t\mapsto G(t-1)$, that is the probability distribution of $\ln(\mathcal{X})+1$. Defining the measure $\nu$ by $\nu(A) = \int_{A} e^{(t+\frac{1}{2})^2/2}dG(t)$, we get $\nu(A-1) = \int_A e^{(t-\frac{1}{2})^2/2}dG(t-1) = \int_A e^{(t-\frac{1}{2})^2/2} e^t dG(t) = \int_A e^{(t+\frac{1}{2})^2/2} dG(t) = \nu(A)$. Hence $\nu$ is $1$-periodic and satisfies \eqref{eq:periodicmeasure}.

As for \eqref{eq:lognormalmoments}, we have
\begin{equation*}
\begin{split}
\int_0^\infty x^k dF(x) &= \int_{-\infty}^\infty e^{kt}\,dG(t)=\int_{-\infty}^\infty e^{kt} e^{-(t+\frac{1}{2})^2/2}\,d\nu(t)\\
&=e^{k(k-1)/2} \int_{-\infty}^\infty e^{-(t-k+\frac{1}{2})^2/2}\,d\nu(t)\\
&=e^{k(k-1)/2} \int_{-\infty}^\infty e^{-(t+\frac{1}{2})^2/2}\,d\nu(t),
\end{split}
\end{equation*}
where the integral on the last line evaluates to one as $\int_{-\infty}^\infty e^{-(t+\frac{1}{2})^2/2}\,d\nu(t) = \lim_{t\rightarrow\infty} G(t)$ = 1.
\end{proof}

We now turn to the proof of Theorem \ref{thm:hardrel}. For any Hamiltonian path $P$, let $X_P$ denote the indicator that $P$ is increasing. Hence $X=\sum_P X_P$. Computing the expectation of $Y$ from the definition, we get
\begin{equation}\label{eq:EY}
\begin{split}
\E Y &= \sum_A \E\left[X_A \middle\vert X_{P_0}=1\right]\\
&= \frac{1}{ \Pr\left(X_{P_0}=1\right)} \sum_A \E X_{P_0} X_A \\
&= (n-1)! \frac{1}{n!} \sum_P  \sum_A \E X_P X_A = \frac{1}{n}\E X^2.
\end{split}
\end{equation}
As mentioned in the introduction, Lavrov and Loh \cite{LL16} showed that the second moment of $X$ is $\sim e n^2$. This implies that $$\E Y \sim e \E X$$ (as expected if $Y$ is approximately $X$ scaled up by a factor of $e$).

In estimating the difference between $eX$ and $Y$, it turns out to be useful to introduce a third random variable, $Z=Z_n$, the number of increasing Hamiltonian paths that are edge disjoint from $P_0$. We similarly have that $Z$ approximately equals $X$ up to a constant factor (in this case $e^{-2}$). For any two paths $A, B$ in $K_n$, let $\abs{A\cap B}$ denote the number of edges they have in common. By linearity of expectation, we have
\begin{equation*}
\E\left[Z\middle\vert X\right] = \sum_{A} \mathbbm{1}_{\abs{A\cap P_0}=0} \E\left[ X_A \middle\vert X\right] =  \sum_A \mathbbm{1}_{\abs{A\cap P_0}=0} \frac{X}{n!} = \left(\frac{1}{n!} \sum_A \mathbbm{1}_{\abs{A\cap P_0}=0}\right) X.
\end{equation*}

Let us label the vertices of $K_n$ by $1, 2, \dots n$ such that $P_0$ corresponds to the vertex sequence $\{1, 2, \dots, n\}$. Then we can interpret the sum $\frac{1}{n!}\sum_A \mathbbm{1}_{\abs{A\cap P_0}=0}$ as the probability that a uniformly chosen permutation $A=\{a_1, a_2, \dots a_n\}$ of $1, 2, \dots, n$ satisfies that $\abs{a_i-a_{i+1}} > 1$ for all $i=1, \dots, n-1$. It was shown, although somewhat sketchily, by Wolfowitz\footnote{In the terminology of Wolfowitz, a permutation with this property is said to have $n$ runs of consecutive elements.} \cite{W44} and elaborated on by Kaplansky \cite{K45} that this tends to $e^{-2}$ as $n\rightarrow \infty$. It follows that 
$$\E Z = \E \E[Z\vert X] \sim e^{-2} \E X = e^{-2}n,$$
(again, this is as expected if $Z$ approximately equals $e^{-2}X$).


For any real numbers $a$ and $b$ we have $2a^2+2b^2 -(a+b)^2 = a^2-2ab+b^2 = (a-b)^2\geq 0$. Hence $(a+b)^2 \leq 2a^2+2b^2$. By plugging in $a=Y-e^3Z$, $b=e^3Z-eX$ and taking expectation it follows that
$\E\left[\left(Y-e\,X\right)^2\right] \leq 2\E\left[\left(Y-e^3\,Z\right)^2\right]+2\E\left[\left(e^{3}Z-e\,X\right)^2\right]$. Hence, in order to prove Theorem \ref{thm:hardrel}, it suffices to show that
\begin{align*}
&\E\left[\left(Y-e^3Z\right)^2\right] = \E Y^2 - 2e^{3}\E YZ + e^{6} \E Z^2=o(n^2),\\
\intertext{and}
&\E\left[\left(e^2Z-X\right)^2\right] = e^{4}\E Z^2 - 2e^2\E XZ + \E X^2=o(n^2).
\end{align*}

Some of the terms in the two equations above we already know how to estimate. As already stated, $\E X^2 \sim e n^2$. We further have $\E XZ = \E [X \E[Z\vert X]] \sim e^{-2}\E X^2 \sim e^{-1} n^2$. Following the argument in \eqref{eq:EY}, we have
\begin{equation}
\E Y^2 =\sum_{A}\sum_B \E\left[ X_A X_B \middle\vert X_{P_0}=1\right]=\frac{1}{n} \sum_P \sum_{A}\sum_B \E\left[ X_{P}X_AX_B \right]=\frac{1}{n} \E X^3,
\end{equation}
so estimating the second moment of $Y$ is equivalent to estimating the third moment of $X$. Similarly, $\E YZ$ and $\E Z^2$ can be expressed as sums over subsets of the terms in $\E X^3$. More precisely, as $Z$ does not depend on the ordering of edges along $P_0$, it does not matter if it is defined over the original edge-ordering for $X$ or the modified one for $Y$, hence $YZ$ has the conditional distributions of $XZ$ given the event that $P_0$ is increasing. Therefore,
\begin{equation*}
\E YZ = \E[XZ\vert X_{P_0}=1] = \frac{1}{n}\sum_P \sum_A \sum_{\abs{B\cap P}=0} \E\left[ X_PX_AX_B \right].
\end{equation*}
Moreover,
\begin{equation*}
\E Z^2 = \E[Z^2\vert X_{P_0}=1] =  \frac{1}{n}\sum_P \sum_{\abs{A\cap P}=0} \sum_{\abs{B\cap P}=0} \E\left[ X_PX_AX_B \right].
\end{equation*}

Hence, in order to prove Theorem \ref{thm:hardrel}, it remains to show the following third moment estimates on $X$.
\begin{proposition}\label{prop:thirdmomentestimates}
As $n\rightarrow\infty$, we have
\begin{equation}\label{eq:EX3}
\E X^3 = \sum_A \sum_B \sum_C \E X_AX_BX_C \sim e^{3} n^3,
\end{equation}
and furthermore,
\begin{align}
\sum_A \sum_B \sum_{\abs{C\cap A}=0} \E X_AX_BX_C &\sim n^3,\label{eq:EX3CdisjointA}\\
\sum_A \sum_{\abs{B\cap A}=0} \sum_{\abs{C\cap A}=0} \E X_AX_BX_C &\sim e^{-3} n^3\label{eq:EX3BCdisjointA}.
\end{align}
\end{proposition}
The proof of this proposition will be given in Section \ref{sec:proofofEX3est}. This completes the proof of our main result.

\begin{remark}
To end this section, we wish to point out two connections to already existing literature. First, the idea that one can couple a random variable to its size-biased version in order to determine properties of its distribution is also commonly used in applications of Stein's method, see for instance Section 3.4 in \cite{R11}. It might be interesting to see if there is a deeper connection here. Perhaps this is a possible avenue to resolving Conjecture \ref{conj:lognormal}. Second, random variables where size-biasing scales up the distribution by a constant factor has been investigated before, see Section 8 in \cite{AGK15+}, where similar properties to Theorem \ref{thm:limitpointprops} are derived. However, to the authors knowledge, the current paper is the first time this observation has been applied in the combinatorics literature.
\end{remark}

\section{Third moment analysis\label{sec:proofofEX3est}}

The proof of Proposition \ref{prop:thirdmomentestimates} will be structured as follows. We start by reformulating the third moment quantities in question as sums over \emph{edge-ordered triples} of Hamiltonian paths. To describe the possible ways the three paths may overlap, we define the \emph{common edge graph} and \emph{reduced common edge graph} of a triple of paths. After giving some combinatorial estimates for these objects, we will prove in Proposition \ref{prop:EX3On3} that $\E X^3=O(n^3)$ (but without obtaining the optimal constant). As a consequence of this analysis, we are able to identify the dominating contribution to $\E X^3$. Proposition \ref{prop:termwiselimit} gives more precise estimates for the dominating terms. This allows us to estimate equations \eqref{eq:EX3}-\eqref{eq:EX3BCdisjointA} as desired.

In what follows, we will use $A, B$ and $C$ to denote (directed) Hamiltonian paths. For such a path $A$, $X_A$ is the indicator for the event that the edges in $A$ are in ascending order. We denote by $\abs{A\cap B}$ and $\abs{ (A\cup B) \cap C}$ etc. the number of edges of the respective sets. Starting from the equality
\begin{equation*}
\E X^3 = \sum_{A, B, C} \E X_AX_BX_C = \sum_{A, B, C} \Pr(A, B, C \text{ are all increasing}),
\end{equation*}
we note that we $\Pr(A, B, C \text{ are all increasing})=\sum_{\preceq} \frac{1}{\abs{A\cup B\cup C}!}$ where $\preceq$ goes over all orderings of the edges of $A\cup B\cup C$ such that $A, B$ and $C$ are all increasing. Hence, we may equivalently write $\E X^3$ as a sum over all $(A, B, C, \preceq)$ where $\preceq$ is an ordering of the edges of $A\cup B\cup C$ that makes all three paths increasing:
\begin{align*}
\eqref{eq:EX3} &= \sum_{(A, B, C, \preceq)} \frac{1}{\abs{A\cup B\cup C}!}.
\intertext{Similarly we get the sums}
\eqref{eq:EX3CdisjointA}&=\sum_{(A, B, C, \preceq)} \frac{\mathbbm{1}_{\abs{A\cap C}=0}}{\abs{A\cup B\cup C}!},
\intertext{and}
\eqref{eq:EX3BCdisjointA} &= \sum_{(A, B, C, \preceq)} \frac{\mathbbm{1}_{\abs{A\cap B}=\abs{A\cap C}=0}}{\abs{A\cup B\cup C}!}.
\end{align*}
For Hamiltonian paths $A, B, C$ and an ordering $\preceq$ of the edges of $A\cup B\cup C$, we say that $(A, B, C, \preceq)$ is an \emph{edge-ordered triple} if $A, B$ and $C$ are all increasing with respect to $\preceq$. Below, we will sometimes for brevity suppress the notation $\preceq$ and simply denote an edge-ordered triple by $(A, B, C)$.


Dealing with edge-ordered triples $(A, B, C, \preceq)$, an important concept is the corresponding \emph{common edge graph}, $G=G(A, B, C, \preceq)$. This is defined as follows. Take the induced subgraph of $K_n$ consisting of all edges that are contained in at least two of the paths (any vertex of degree zero is removed). Label each edge according to which paths the edge is contained in, and in which direction these traverse the edge. We consider $G$ as an edge-ordered graph by letting it inherit the order $\preceq$ of $(A, B, C)$. See Figure \ref{fig:ceg} for an example.


\begin{figure}
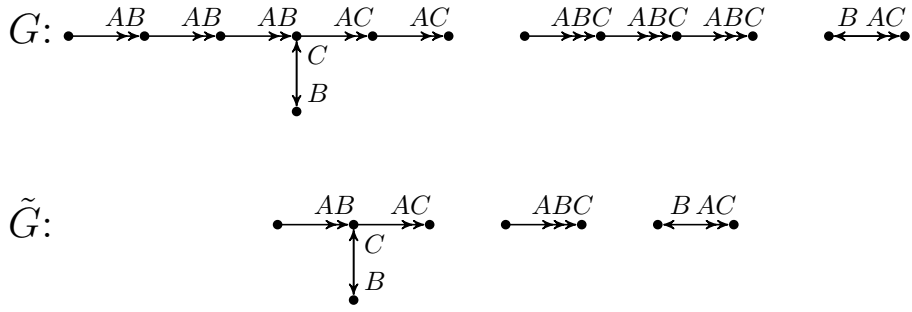

\begin{center}
\tikz [>=stealth',shorten >=2pt,auto,node distance=1.4cm,semithick,baseline=-3pt,font=\fontsize{10}{10}\selectfont]{
\node[font=\fontsize{16}{16}\selectfont] at (.5,0.1) {$G$:};
\node[font=\fontsize{16}{16}\selectfont] at (.5,-1.9-.5) {$\tilde{G}$:};
\draw (1,0) edge[->>] node[above,near end,font=\small] {$AB$} (2,0);
\draw (2,0) edge[->>] node[above,near end,font=\small] {$AB$} (3,0);
\draw (3,0) edge[->>] node[above,near end,font=\small] {$AB$} (4,0);
\draw (4,0) edge[->>] node[above,near end,font=\small] {$AC$} (5,0);
\draw (5,0) edge[->>] node[above,near end,font=\small] {$AC$} (6,0);
\draw (4, 0) -- node[right,near end,font=\small] {$B$} node[right,near start,font=\small] {$C$} (4, -1);
\draw (4, -.07) edge[<->] (4, -1);
\draw (7,0) edge[->>>] node[above,near end,font=\small] {$ABC$} (8,0);
\draw (8,0) edge[->>>] node[above,near end,font=\small] {$ABC$} (9,0);
\draw (9,0) edge[->>>] node[above,near end,font=\small] {$ABC$} (10, 0);
\draw (11,0) -- node[above,near start,font=\small] {$B$} node[above,near end,font=\small] {$AC$} (12, 0);
\draw (11.07,0) edge[<->>] (12,0);
\draw[draw=black,fill=black] (1,0) circle (1.5pt);
\draw[draw=black,fill=black] (2,0) circle (1.5pt);
\draw[draw=black,fill=black] (3,0) circle (1.5pt);
\draw[draw=black,fill=black] (4,0) circle (1.5pt);
\draw[draw=black,fill=black] (5,0) circle (1.5pt);
\draw[draw=black,fill=black] (6,0) circle (1.5pt);
\draw[draw=black,fill=black] (7,0) circle (1.5pt);
\draw[draw=black,fill=black] (8,0) circle (1.5pt);
\draw[draw=black,fill=black] (9,0) circle (1.5pt);
\draw[draw=black,fill=black] (10,0) circle (1.5pt);
\draw[draw=black,fill=black] (11,0) circle (1.5pt);
\draw[draw=black,fill=black] (12,0) circle (1.5pt);
\draw[draw=black,fill=black] (4,-1) circle (1.5pt);
\draw (1+2.75, -2-.5) edge[->>] node[above,near end,font=\small] {$AB$} (2+2.75, -2-.5);
\draw (2+2.75, -2-.5) edge[->>] node[above,near end,font=\small] {$AC$} (3+2.75, -2-.5);
\draw (2+2.75, -2-.5) -- node[right,near end,font=\small] {$B$} node[right,near start,font=\small] {$C$} (2+2.75, -3-.5);
\draw (2+2.75, -2.07-.5) edge[<->] (2+2.75, -3-.5);
\draw (4+2.75, -2-.5) edge[->>>] node[above,near end,font=\small] {$ABC$} (5+2.75, -2-.5);
\draw (6.07+2.75, -2-.5) edge[<->>] node[above,near start,font=\small] {$B$} node[above,near end,font=\small] {$AC$} (7+2.75, -2-.5);
\draw[draw=black,fill=black] (1+2.75,-2-.5) circle (1.5pt);
\draw[draw=black,fill=black] (2+2.75,-2-.5) circle (1.5pt);
\draw[draw=black,fill=black] (3+2.75,-2-.5) circle (1.5pt);
\draw[draw=black,fill=black] (4+2.75,-2-.5) circle (1.5pt);
\draw[draw=black,fill=black] (5+2.75,-2-.5) circle (1.5pt);
\draw[draw=black,fill=black] (6+2.75,-2-.5) circle (1.5pt);
\draw[draw=black,fill=black] (7+2.75,-2-.5) circle (1.5pt);
\draw[draw=black,fill=black] (2+2.75,-3-.5) circle (1.5pt);
}
\end{center}
\caption{\label{fig:ceg} An example of a common edge graph $G$ and the corresponding reduced common edge graph $\tilde{G}$. In both cases we consider the edges as being ordered from left to right, with the vertical ``BC'' edges placed between the incident ``AB'' and ``AC'' edges. Note that the components of $G$ and $\tilde{G}$ need not be paths. (In fact, they can even contain cycles.)}
\end{figure}



\begin{claim}\label{claim:GtildeGorder}
Let $G$ be a common edge graph of some edge-ordered triple, and let $v_0, e_1, v_1, e_2, \dots,$ $v_{l-1}, e_l, v_l$ be a path in $G$ where all edges have the same label, i.e. they are all contained in one of $(A\cap B)\setminus C$, $(A\cap C)\setminus B$, $(B\cap C)\setminus A$ or $A\cap B \cap C$. Then $e_1, e_2, \dots, e_l$ are next to each other in the edge-ordering of the common edge graph. Moreover, for any $0<i<l$, the only edges incident to $v_i$ in $G$ are $e_{i}$ and $e_{i+1}$.
\end{claim}
\begin{proof}
Let us, without loss of generality, assume that the edges are in $A\cap B$. As $A, B,$ and $C$ are self-avoiding, no edges in $A$ or $B$ can come between $e_i$ and $e_{i+1}$ in the edge-ordering. Hence any such edge is unique to $C$, and not in the common edge graph. Similarly $e_i$ and $e_{i+1}$ are the only edges in $A$ and $B$ respectively incident to $v_i$, which makes them the only incident edges in the common edge graph. 
\end{proof}

Motivated by this claim, we define the \emph{reduced common edge graph} of an edge-ordered triple, $\tilde{G}=\tilde{G}(A, B, C, \preceq)$, as the graph obtained by collapsing all paths as above in the common edge graph to a single edge, preserving edge labels and the edge-ordering, see Figure \ref{fig:ceg}. It follows from the preceding claim that $G$, including edge-ordering, can be uniquely recovered up to isomorphism from $\tilde{G}$ if one knows the lengths each of the collapsed paths.

For two paths $A$ and $B$, we will refer to a connected component of $A\cap B$ as a \emph{common segment} of $A$ and $B$. Moreover, for three paths $A, B$ and $C$, a common segment of $A\cup B$ and $C$ is a connected component of $(A\cup B)\cap C$. Note that any such component is a connected subgraph of $C$, hence a common segment is always a path segment. For an edge-ordered triple $(A, B, C)$, we define
\begin{align}
\label{eq:defk1}k_1 = &\text{ the number of common segments between $A$ and $B$},\\
\label{eq:defk2}k_2 = &\text{ the number of common segments between $A\cup B$ and $C$},\\
\label{eq:defk3}k_3 = &\text{ the number of components in the common edge graph}\\\nonumber &\text{ that do not contain an edge common to $A$ and $B$},\\
\label{eq:defk4}k_4 = &\text{ the number of vertices in the common edge graph which}\\\nonumber &\text{ are incident to an edge in $A\cap C$ and one in $B\cap C$, but}\\\nonumber &\text{ none in $A\cap B$}.
\end{align}
Note that each of these quantities are uniquely determined by the reduced common edge graph.

\begin{figure}[h!]
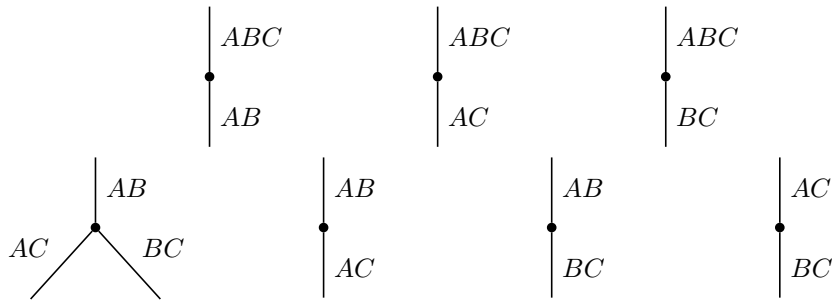

\begin{center}
\tikz [>=stealth',shorten >=2pt,auto,node distance=1.4cm,semithick,baseline=-3pt,font=\fontsize{10}{10}\selectfont]{
\node[draw=black,fill=black,circle,inner sep=0pt,minimum size=3pt] (ABCAB) at (-3,0){};
\node[draw=black,fill=black,circle,inner sep=0pt,minimum size=3pt] (ABCAC) at (0,0){};
\node[draw=black,fill=black,circle,inner sep=0pt,minimum size=3pt] (ABCBC) at (3,0){};
\node[draw=black,fill=black,circle,inner sep=0pt,minimum size=3pt] (ABACBC) at (-4.5,-2){};
\node[draw=black,fill=black,circle,inner sep=0pt,minimum size=3pt] (ABAC) at (-1.5,-2){};
\node[draw=black,fill=black,circle,inner sep=0pt,minimum size=3pt] (ABBC) at (1.5,-2){};
\node[draw=black,fill=black,circle,inner sep=0pt,minimum size=3pt] (ACBC) at (4.5,-2){};
\draw (ABCAB) -- +(90:1) node [midway, right] {$ABC$};
\draw (ABCAB) -- +(270:1) node [midway, right] {$AB$};
\draw (ABCAC) -- +(90:1) node [midway, right] {$ABC$};
\draw (ABCAC) -- +(270:1) node [midway, right] {$AC$};
\draw (ABCBC) -- +(90:1) node [midway, right] {$ABC$};
\draw (ABCBC) -- +(270:1) node [midway, right] {$BC$};
\draw (ABACBC) -- +(90:1) node [midway, right] {$AB$};
\draw (ABACBC) -- +(-.9,-1) node [midway, above left] {$AC$};
\draw (ABACBC) -- +(.9,-1) node [midway, above right] {$BC$};
\draw (ABAC) -- +(90:1) node [midway, right] {$AB$};
\draw (ABAC) -- +(270:1) node [midway, right] {$AC$};
\draw (ABBC) -- +(90:1) node [midway, right] {$AB$};
\draw (ABBC) -- +(270:1) node [midway, right] {$BC$};
\draw (ACBC) -- +(90:1) node [midway, right] {$AC$};
\draw (ACBC) -- +(270:1) node [midway, right] {$BC$};
}
\end{center}
\caption{\label{fig:posneighbs} The possible neighborhoods of a vertex with degree at least $2$ in a reduced common edge graph. Note that the edge labels of $G$ should contain both information about which paths traverse each edge, and in which directions these paths traverse it. However, to reduce the number of cases this latter part of the labels are not indicated here.}
\end{figure}

\begin{claim}\label{claim:regbound}
Let $k_1, k_2$ and $k_4$ be given non-negative integers. Then (up to isomorphism) the number of reduced common edge graphs $\tilde{G}$ such that $k_1(\tilde{G})=k_1, k_2(\tilde{G})=k_2$ and $k_4(\tilde{G})=k_4$ is at most $e^{O(k_1+k_2+k_4)}$. Moreover, any such $\tilde{G}$ contains at most $O(k_1+k_2+k_4)$ edges.
\end{claim}
\begin{proof}
We first bound the number of vertices in $\tilde{G}$. Consider the possible sets of edges incident to a vertex $v\in\tilde{G}$. If $v$ has degree $1$, then its incident edge is either common to $A$ and $B$, in which case it is an end-point of a common segment of $A$ and $B$ and hence contributes by $\frac{1}{2}$ to $k_1$, or the incident edge is common to $C$ and at least one of $A$ and $B$, in which case it is an end-point of a common segment of $A\cup B$ and $C$, and hence contributes by $\frac{1}{2}$ to $k_2$. The cases when $v$ has degree at least $2$ is illustrated in Figure \ref{fig:posneighbs}. To reduce the number of cases, we can ignore the directions in which the paths traverse the various edges, and only distinguish which paths each edge is contained in. We see clearly that $v$ is either the end-point of a common segment of $A$ and $B$ (cases $2$, $3$, $4$, $5$ and $6$), in which case it contributes by $\frac{1}{2}$ to $k_1(\tilde{G})$, an end-point of a common segment of $A\cup B$ and $C$ (cases $1$, $5$ and $6$), in which case it contributes $\frac{1}{2}$ to $k_2(\tilde{G})$, or counted in $k_4(\tilde{G})$ (case $7$). As any vertex has at most three incident edges, it follows that $\abs{E(\tilde{G})}=O(k_1+k_2+k_4)$.

Suppose that we construct $\tilde{G}$ by adding the edges one at a time, in ascending order. For each new edge $e$ we need to choose its label, whether or not its end-points are already in the graph, and in that case which vertices these are. Assume that $e$ has an end-point already in the graph. Since we are inserting edges sorted by priority, $e$ can be adjacent only to an edge with the highest order among the edges which are part of the same path. For any such edge $e'$ and for a fixed orientation of $e$, there is only one possible vertex of $e'$ where we can attach $e$. This gives at most three possible vertices for attachment for $e$. Thus $\tilde{G}$ can be encoded in $O(\abs{E(\tilde{G})})$ bits.
\end{proof}

\begin{claim}\label{claim:togetherorapart}
Let $(A, B, C)$ be an edge-ordered triple in $K_n$. Then, for any $\varepsilon>0$, either $\abs{(A' \cup B') \cap C'}\leq (1-\varepsilon)n$ for some permutation $(A', B', C')$ of $(A, B, C)$, or $\abs{A\cap B \cap C}\geq (1-18\varepsilon)n$.
\end{claim}
\begin{proof}
Let $G$ be the common edge graph of $(A, B, C)$. We define the weight of a vertex in $G$, denoted by $w(v)$, as the number of edges incident to $v$ in $G$, where an edge is counted with multiplicity $2$ if it is shared by two paths and $3$ if shared by all three. It is not too hard to see that the only ways a vertex can have the maximal weight of $6$ is 
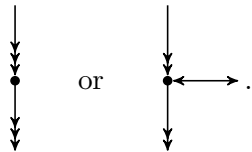
\begin{equation*}
\hbox{\tikz [>=stealth',shorten >=2pt,auto,node distance=1.4cm,semithick,baseline=-3pt]{
\draw (0, 1) edge[->>>] (0, 0);
\draw (0, 0) edge[->>>] (0, -1);
\draw[draw=black,fill=black] (0,0) circle (1.5pt);
}}\qquad\text{or}\qquad\hbox{\tikz [>=stealth',shorten >=2pt,auto,node distance=1.4cm,semithick,baseline=-3pt]{
\draw (0, 1) edge[->>] (0, 0);
\draw (.07, 0) edge[<->] (1, 0);
\draw (0, 0) edge[->>] (0, -1);
\draw[draw=black,fill=black] (0,0) circle (1.5pt);
}}.
\end{equation*}
Let us denote the number of these types of vertices by $x$ and $y$ respectively.

Now, one can observe that no \hbox{\tikz [>=stealth',shorten >=2pt,auto,node distance=1.4cm,semithick,baseline=-3pt]{
\draw (.07, 0) edge[<->] (1, 0);
}} edge can have two end-points of weight $6$, as it is impossible to order the edges in
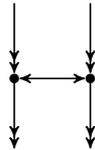
\begin{equation*}
\hbox{\tikz [>=stealth',shorten >=2pt,auto,node distance=1.4cm,semithick,baseline=-3pt]{
\draw (0, 1) edge[->>] (0, 0);
\draw (.07, 0) edge[<->] (1, 0);
\draw (0, 0) edge[->>] (0, -1);
\draw (1, 1) edge[->>] (1, 0);
\draw (1, 0) edge[->>] (1, -1);
\draw[draw=black,fill=black] (0,0) circle (1.5pt);
\draw[draw=black,fill=black] (1,0) circle (1.5pt);
}}
\end{equation*}
such that the path segments of $A$, $B$ and $C$ are all increasing. Furthermore, as each of the $n-x-y$ remaining vertices can have at most two \hbox{\tikz [>=stealth',shorten >=2pt,auto,node distance=1.4cm,semithick,baseline=-3pt]{
\draw (.07, 0) edge[<->] (1, 0);
}} edges, it follows that $y\leq\text{ number of }$ \hbox{\tikz [>=stealth',shorten >=2pt,auto,node distance=1.4cm,semithick,baseline=-3pt]{
\draw (.07, 0) edge[<->] (1, 0);
}} $\text{ edges } \leq 2(n-x-y)$, which implies that $y \leq \frac{2}{3}(n-x)$.

Now, assume that $\abs{A\cap B\cap C} < (1-18\varepsilon)n$. One readily sees that $\abs{A\cap B \cap C} \geq x$. We then have
\begin{equation*}
\begin{split}
&\abs{(A\cup B) \cap C} + \abs{(A\cup C) \cap B} + \abs{(B\cup C) \cap A}\\
&\qquad = \frac{1}{2} \sum_{v\in G} w(v) \leq 3(x+y)+\frac{5}{2}(n-x-y)\\
&\qquad \leq \frac{17n}{6}+\frac{x}{6}\leq 3(1-\varepsilon)n,
\end{split}
\end{equation*}
where in the second to last step we used $y \leq \frac{2}{3}(n-x)$.
\end{proof}

Let $\tilde{G}$ be a reduced common edge graph, and let $c_{AB}, c_{AC}, c_{BC}$ and $c_{ABC}$ be non-negative integers. To simplify notation we will write $\bar{c}=(c_{AB}, c_{AC}, c_{BC}, c_{ABC})$. Let $\Trip_n(\tilde{G}, \bar{c})$ denote the number of edge-ordered triples $(A, B, C)$ corresponding to $\tilde{G}$ such that $$\abs{ (A \cap B) \setminus C}= c_{AB},\qquad\abs{ (A \cap C) \setminus B}= c_{AC},\qquad\abs{ (B \cap C) \setminus A}= c_{BC}$$ and $$\abs{ A \cap B \cap C}= c_{ABC}.$$
Note that this means that $\abs{A\cup B\cup C} = 3n-3-c_{AB}-c_{AC}-c_{BC}-2c_{ABC}$. Hence, letting $\trip_n(\tilde{G}, \bar{c}) = \Trip_n(\tilde{G}, \bar{c})/(3n-3-c_{AB}-c_{AC}-c_{BC}-2c_{ABC})!$, we can rewrite the weighted sum over edge-ordered triples in the beginning of the section as $\E X^3 = \sum_{\tilde{G}} \sum_{\bar{c}} t_n(\tilde{G}, \bar{c})$.

The following proposition gives a reasonably good bound for $\Trip_n(\tilde{G}, \bar{c})$ provided either $c_{ABC}$ is close to $n$, or $n-c_{AC}-c_{BC}-c_{ABC}$ is of order $n$. By Claim \ref{claim:togetherorapart} we know that we can always permute the paths in an edge-ordered triple such that one of these two properties is satisfied.

\begin{proposition}\label{prop:trip}
For $\tilde{G}$ and $\bar{c}$ as above, we have
\begin{equation*}
\begin{split}
&\Trip_n(\tilde{G}, \bar{c}) \leq \left(\prod_\alpha {c_\alpha-1 \choose l_\alpha-1}\right) 
{2n-2-2c_{AB}-2c_{ABC}+k_1 \choose n-1-c_{AB}-c_{ABC}, n-1-c_{AB}-c_{ABC}, k_1}
\cdot\\
&\cdot{2n-2-2c_{AB}-2c_{ABC} \choose k_3}{3n-3-c_{AB}-2c_{AC}-2c_{BC}-3c_{ABC}+k_2 \choose n-1-c_{AC}-c_{BC}-c_{ABC} }\cdot\\
&\cdot n! (n-c_{AB}-c_{ABC}-k_1-k_4)! (n-c_{AC}-c_{BC}-c_{ABC}-k_2)!,
\end{split}
\end{equation*}
where the product goes over $\alpha\in\{AB, AC, BC, ABC\}$, $k_i=k_i(\tilde{G})$ for $i=1, 2, 3, 4$ and $l_\alpha = l_\alpha(\tilde{G})$ denotes the number of edges in $\tilde{G}$ labelled $\alpha$, that is the number of $AB$-, $AC$-, $BC$-, and $ABC$-edges respectively in $\tilde{G}$.
\end{proposition}
\begin{remark}
Note that $c_{AB}$, $c_{AC}$, $c_{BC}$ and $c_{ABC}$ are allowed to be $0$. In that case, we should think of the corresponding binomial factor ${c_\alpha -1 \choose l_\alpha - 1}$ as $1$ if $l_\alpha=0$ and $0$ otherwise, which is consistent with the interpretation of ${c_\alpha -1 \choose l_\alpha-1}$ as the number of ways to place $c_\alpha$ indistinguishable balls in $l_\alpha$ bins such that each bin contains at least one ball.
\end{remark}
\begin{proof}
In counting the number of edge-ordered triples $(A, B, C, \preceq)$ satisfying the conditions above, it is useful to consider what we refer to as an \emph{edge list}, first representing $A$ and $B$, and then extended to represent $A, B$ and $C$. The edge list of $A$ and $B$ is defined as a sequence of length $\abs{A\cup B}$ with elements ``\textsf{A}'', ``\textsf{B}'' and ``\textsf{AB}'' where the $i$:th character denotes which of $A$ and $B$ the $i$:th lowest ordered edge among $A\cup B$ is contained in. Similarly, the edge list of $A$, $B$ and $C$ is a sequence of length $\abs{A\cup B\cup C}$ with elements ``\textsf{A}'', ``\textsf{B}'', ``\textsf{C}'', ``\textsf{AB}'', ``\textsf{AC}'', ``\textsf{BC}'' and ``\textsf{ABC}'' where the $i$:th character denotes which of $A$, $B$ and $C$ the $i$:th lowest ordered edge among $A\cup B\cup C$ is contained in. This serves both to encode the edge-ordering of $A\cup B\cup C$ as well as the number of unique edges the paths have between common segments.
\\~\\
\textsc{Step 1:} Choose the common edge graph $G$.
\\~\\
It suffices to choose the lengths of each collapsed path in $\tilde{G}$. By Claim \ref{claim:GtildeGorder} we know that this uniquely defines the edge-ordering of $G$. For each $\alpha \in \{AB, AC, BC, ABC\}$ we need to divide up $c_\alpha$ edges between $l_\alpha$ segments. Hence, there are at most $\prod_\alpha {c_\alpha-1 \choose l_\alpha-1}$ ways to do this.
\\~\\
\textsc{Step 2:} Choose the edge list of $A$ and $B$.
\\~\\
Given $G$, we know the number of common segments of $A$ and $B$ and their respective lengths. As no edge unique to $A$ or $B$ can occur during one of the common segments, any edge list can, given $G$, be encoded as a string containing $\abs{A\setminus B} = n-1-c_{AB}-c_{ABC}$ \textsf{A}:s, $\abs{B\setminus A} = n-1-c_{AB}-c_{ABC}$ \textsf{B}:s, and $k_1$ \textsf{D}:s (where each \textsf{D} acts as a placeholder for a sequence of incident \textsf{AB}:s). Hence the number of such edge lists is at most the multinomial coefficient,
${2n-2-2c_{AB}-2c_{ABC}+k_1 \choose n-1-c_{AB}-c_{ABC}, n-1-c_{AB}-c_{ABC}, k_1}.$
\\~\\
\textsc{Step 3:} For each edge in $G$, locate the corresponding entry in the edge list of $A$ and $B$.
\\~\\
Note that each edge in $G$ is contained in at least one of $A$ and $B$, and is hence present in the edge list. The position of any edge common to $A$ and $B$ is immediately determined by the edge-ordering of $G$ -- the $i$:th common edge between $A$ and $B$ in $G$ corresponds to the $i$:th \textsf{AB} in the edge list. Moreover, if we know which position one edge in $G$ has in the edge list, then we also know the position any adjacent edge in $G$ as such an edge is either the previous or next edge in one of $A$ and $B$. Thus, the position of one edge implies the positions of all edges in the same component. Hence, it suffices to choose the location of one edge from each of the $k_3$ components in $G$ that does not contain a common edge of $A$ and $B$. As any such edge cannot be common to $A$ and $B$, this can be done in ${\abs{A\triangle B} \choose k_3} = {2n-2-2c_{AB}-2c_{ABC} \choose k_3}$ ways.
\\~\\
\textsc{Step 4:} Extend the edge list of $A$ and $B$ to an edge list of $A$, $B$ and $C$.
\\~\\
As we have already identified which positions in the edge list of $A$ and $B$ that correspond to edges common between $C$ and at least one of $A$ and $B$, we can just add a \textsf{C} to any such position. It  remains to insert $\abs{C\setminus(A\cup B)}=n-1-c_{AC}-c_{BC}-c_{ABC}$ many \textsf{C}:s into this list. There are at least $c_{AC}+c_{BC}+c_{ABC}-k_2$ pairs of adjacent elements in the list between which we cannot place any \textsf{C}:s as these correspond to adjacent edges in common segments between $A\cup B$ and $C$. Hence, the number of ways this extension can be made is at most ${\abs{A\cup B}+\abs{C\setminus(A\cup B)}-(c_{AC}+c_{BC}+c_{ABC}-k_2) \choose \abs{C\setminus(A\cup B)}} = {3n-3-c_{AB}-2c_{AC}-2c_{BC}-3c_{ABC}+k_2 \choose n-1-c_{AC}-c_{BC}-c_{ABC} }.$
\\~\\
\textsc{Step 5:} Choose the vertex sequences of $A, B$ and $C$.
\\~\\
There are $n!$ possibilities for the vertex sequence of $A$. Given $G$ and the edge-ordering of $\{A, B, C\}$ this determines $c_{AB}+c_{ABC}+k_1+k_4$ of the vertices along $B$, yielding at most $(n-c_{AB}-c_{ABC}-k_1-k_4)!$ options for the remaining vertices of $B$. Similarly, fixing $A$ and $B$, the remaining vertices along $C$ can be chosen in at most $(n-c_{AC}-c_{BC}-c_{ABC}-k_2)!$ ways.
\end{proof}

\begin{lemma}\label{lemma:factorialestimates}
For any non-negative integers $p\leq q \leq r$ where $r\geq 1$, we have
\begin{equation*}
\frac{(r-p)!}{(q-p)!} \leq \frac{r!}{q!} \left( \frac{q}{r}\right)^p \leq \frac{r!}{q!} \exp\left( -\frac{p(r-q)}{r}\right),
\end{equation*}
and
\begin{equation*}
\frac{(r-p)!}{r!} \leq \left(\frac{e}{r}\right)^p.
\end{equation*}
\end{lemma}
\begin{proof}
We have
\begin{equation*}
\frac{q!}{(q-p)!} \frac{(r-p)!}{r!} = \frac{q-p+1}{r-p+1}\cdot\frac{q-p+2}{r-p+2}\dots \frac{q}{r} \leq \left(\frac{q}{r}\right)^p,
\end{equation*}
where, by convexity of the exponential function, $\frac{q}{r}=1-\frac{r-q}{r}\leq e^{-\frac{r-q}{r}}$. Moreover,
\begin{equation*}
\begin{split}
\ln \left( \frac{(r-p)!}{r!}\right) &= \sum_{t=r-p+1}^r -\ln t\\ &\leq \int_{r-p}^r -\ln t\,dt = (r-p)\ln(r-p) - r \ln r + p\\ &\leq (r-p)\ln r - r\ln r + p = p(1-\ln r). 
\end{split}
\end{equation*}
\end{proof}

We are now ready to derive an upper bound on $\E X^3$ that shows that the quantity is of order $n^3$ and further lets us identify the dominating contribution. We define a reduced common edge graph $\tilde{G}$ as \emph{good} if each of its components consists of a single edge which is common to precisely two paths.

\begin{proposition}\label{prop:EX3On3}
We have $\E X^3 = O(n^3)$.  Moreover, for any integer $M\geq 0$, the contribution to the sum 
\begin{equation*}
\E X^3 = \sum_{\tilde{G}} \sum_{\bar{c}} t_n(\tilde{G}, \bar{c})
\end{equation*}
from all pairs $(\tilde{G}, \bar{c})$ such that either $\tilde{G}$ is not good or $c_{AB}+c_{AC}+c_{BC}>M$ is $O(n^3)e^{-\Omega(M)} + O(n^2)$.
\end{proposition}
\begin{proof}
Fix $\varepsilon>0$ sufficiently small ($\varepsilon \leq 1500000^{-1}$ suffices). By Claim \ref{claim:togetherorapart} we know that we can bound $\E X^3 = \sum_{(A, B, C, \preceq)} \frac{1}{\abs{A\cup B\cup C}!}$ by $3S_1+3S_2+S_3$, where $S_1$ is the contribution to this sum from all edge-ordered triples where $$\abs{A\cap B}\leq (1-\varepsilon)n\qquad\text{and}\qquad\abs{(A\cup B)\cap C}\leq (1-\varepsilon)n,$$ $S_2$ is the contribution where $$\abs{A\cap B} > (1-\varepsilon)n\qquad\text{and}\qquad\abs{(A\cup B)\cap C}\leq (1-\varepsilon)n,$$ and $S_3$ is the contribution where $\abs{A\cap B\cap C}>(1-18\varepsilon)n$.

Let us start by estimating $S_1$. We have
\begin{equation*}
S_1 = \sum_{\tilde{G}} \sum_{\substack{c_{AB}+c_{ABC}\leq (1-\varepsilon)n\\ c_{AC}+c_{BC}+c_{ABC}\leq (1-\varepsilon)n}} \trip_n(\tilde{G}, \bar{c}).
\end{equation*}
To bound the summand on the right we apply Proposition \ref{prop:trip}. Note that we can assume that $c_{AB}+c_{ABC}-k_1 \geq 0$ and $c_{AC}+c_{BC}+c_{ABC}-k_2 \geq 0$, as otherwise $t_n(\tilde{G}, \bar{c})=0$. By the first part of Lemma \ref{lemma:factorialestimates} using $p=c_{AB}+c_{ABC}-k_1$, $q=n-1-k_1$ and $r=2n-2-c_{AB}-c_{ABC}$ we have, for any $c_{AB}, c_{AC}, c_{BC}, c_{ABC}$ as in the above sum,
\begin{equation}
\begin{split}\label{eq:EX3On3binest1}
&{2n-2-2c_{AB}-2c_{ABC}+k_1 \choose n-1-c_{AB}-c_{ABC}, n-1-c_{AB}-c_{ABC}, k_1}\\
&\qquad \leq {2n-2-c_{AB}-c_{ABC} \choose n-1-c_{AB}-c_{ABC}, n-1-k_1, k_1} e^{-\Omega_\varepsilon(c_{AB}+c_{ABC}-k_1)}
\end{split}
\end{equation}
and using $p=c_{AC}+c_{BC}+c_{ABC}-k_2$, $q=2n-2-c_{AB}-c_{ABC}$ and $r=3n-3-c_{AB}-c_{AC}-c_{BC}-2c_{ABC}$
\begin{equation}
\begin{split}\label{eq:EX3On3binest2}
&{3n-3-c_{AB}-2c_{AC}-2c_{BC}-3c_{ABC}+k_2 \choose n-1-c_{AC}-c_{BC}-c_{ABC} }\\
&\qquad \leq {3n-3-c_{AB}-c_{AC}-c_{BC}-2c_{ABC} \choose n-1-c_{AC}-c_{BC}-c_{ABC}} e^{-\Omega_\varepsilon(c_{AC}+c_{BC}+c_{ABC}-k_2)}.
\end{split}
\end{equation}
Here $f=\Omega_\varepsilon(g)$ denotes that there is some constant $M_\varepsilon>0$, possibly depending on $\varepsilon$, such that $f\geq M_\varepsilon g.$

Plugging \eqref{eq:EX3On3binest1} and \eqref{eq:EX3On3binest2} into the estimate for $t_n(\tilde{G}, \bar{c}) = T_n(\tilde{G}, \bar{c})/(3n-3-c_{AB}-c_{AC}-c_{BC}-2c_{ABC})!$ given by Proposition \ref{prop:trip} and cancelling out identical factorials in the numerator and denominator we get
\begin{equation}
\begin{split}\label{eq:tripfactorialsimp}
&t_n(\tilde{G}, \bar{c}) \leq \left( \prod_\alpha {c_\alpha-1 \choose l_\alpha-1} \right){2n-2-2c_{AB}-2c_{ABC} \choose k_3} \frac{e^{-\Omega_{\varepsilon}(c_{AB}+c_{AC}+c_{BC}+2c_{ABC}-k_1-k_2)}}{k_1!} \cdot\\
&\qquad \cdot \frac{n!}{(n-1-k_1)!} \frac{(n-c_{AB}-c_{ABC}-k_1-k_4)!}{(n-1-c_{AB}-c_{ABC})!} \frac{ (n-c_{AC}-c_{BC}-c_{ABC}-k_2)!}{(n-1-c_{AC}-c_{BC}-c_{ABC})!}.
\end{split}
\end{equation}

We have the bounds $${2n-2-2c_{AB}-2c_{ABC} \choose k_3} \leq \frac{(2n)^{k_3}}{k_3!} = \frac{n^{k_3} e^{O(k_3)}}{k_3!},$$ $$\frac{n!}{(n-1-k_1)!} \leq n^{k_1+1}$$ and, by the second part of Lemma \ref{lemma:factorialestimates}, \begin{equation}\frac{(n-c_{AB}-c_{ABC}-k_1-k_4)!}{(n-1-c_{AB}-c_{ABC})!}\leq n^{1-k_1-k_4} e^{O_\varepsilon(k_1+k_4)}\end{equation} and
\begin{equation}\label{eq:factorialAcupBcapC}\frac{ (n-c_{AC}-c_{BC}-c_{ABC}-k_2)!}{(n-1-c_{AC}-c_{BC}-c_{ABC})!} \leq n^{1-k_2} e^{O_\varepsilon(k_2)},
\end{equation}
where $f=O_\varepsilon (g)$ denotes that there exists some constant $M_\varepsilon>0$ such that $\abs{f}\leq M_\varepsilon\abs{g}$. Note that Lemma \ref{lemma:factorialestimates} strictly speaking only applies to the last two inequalities when $k_1+k_4\geq 1$ and $k_2\geq 1$ respectively, but it is clear that the inequalities also hold for $k_1+k_4=0$ and $k_2=0$, as desired.

Plugging this into \eqref{eq:tripfactorialsimp} we get
\begin{equation}\label{eq:S1summandbound}
t_n(\tilde{G}, \bar{c}) = O(n^3) \left( \prod_\alpha{c_\alpha-1\choose l_\alpha-1} e^{-\Omega_\varepsilon(c_\alpha)} \right) \frac{ e^{O_\varepsilon(k_1+k_2+k_3+k_4)} }{k_1! k_3! n^{k_2+k_4-k_3}}.
\end{equation}
Using the Taylor expansion $\frac{z^{l}}{(1-z)^l} = \sum_{c} {c-1\choose l-1} z^c$ which is convergent for any $\abs{z}<1$, it follows that $\sum_{c_\alpha} {c_\alpha-1 \choose l_\alpha-1} e^{-\Omega_\varepsilon(c_\alpha)} \leq e^{O_\varepsilon(l_\alpha)}.$ By the second part of Claim \ref{claim:regbound} we know that $\sum_\alpha l_\alpha = O(k_1+k_2+k_4)$. Hence
\begin{equation}\label{eq:S1summedoverc}
S_1 = O(n^3) \sum_{\tilde{G}} \frac{e^{O_\varepsilon(k_1+k_2+k_3+k_4)}}{k_1!k_3!n^{k_2+k_4-k_3}}.
\end{equation}
By the definitions of $k_2, k_3$ and $k_4$, equations \eqref{eq:defk2}--\eqref{eq:defk4}, one can see that $k_2+k_4-k_3$ is at least the number of components of $\tilde{G}$ that either consist of more than one edge, or contain an $ABC$-edge. In particular, $k_2+k_4-k_3$ is non-negative, and zero only if $\tilde{G}$ is good. Hence, letting $k_5=k_2+k_4-k_3$ and noting that $k_4 \leq k_5+k_3 \leq e^{O_\varepsilon(k_5+k_3)}$, it follows by Claim \ref{claim:regbound} that
\begin{align*}
\sum_{\tilde{G}} \frac{e^{O_\varepsilon(k_1+k_2+k_3+k_4)}}{k_1!k_3!n^{k_2+k_4-k_3}} &\leq \sum_{k_5=0}^\infty \sum_{k_1=0}^\infty \sum_{k_3=0}^\infty  \sum_{k_4=0}^{k_3+k_5} \frac{e^{O_\varepsilon(k_1+k_5+2k_3)}}{k_1!k_3! n^{k_5}}\\
&\leq \left(\sum_{k_5=0}^\infty \frac{e^{O_\varepsilon(2k_5)}}{n^{k_5}}\right)\left(
\sum_{k_1=0}^\infty \frac{ e^{O_\varepsilon(k_1)}}{k_1!}\right)\left( \sum_{k_3=0}^\infty \frac{e^{O_\varepsilon(3k_3)}}{k_3!}\right)=O_\varepsilon(1).
\end{align*}
Hence $S_1=O_\varepsilon(n^3)$. Moreover, bounding the contribution from non-good $\tilde{G}$, in which case we start $k_5$ at $1$, gives us $O_\varepsilon(n^2)$, as desired.


We now turn to $S_2$. Let $d_{AB} = n-c_{AB}-c_{ABC}$. Note that $d_{AB}\geq 1$.  Proceeding as for $S_1$, except for that instead of \eqref{eq:EX3On3binest1} we use the bounds
$${2d_{AB}-2+k_1 \choose d_{AB}-1, d_{AB}-1, k_1} \leq 3^{2d_{AB}-2+k_1},$$
and
$${2d_{AB}-2 \choose k_3}\leq 2^{2d_{AB}-2},$$
we get after cancelling identical factorials in the numerator and denominator that
\begin{align*}
t_n(\tilde{G}, \bar{c}) &\leq \left(\prod_\alpha {c_\alpha -1 \choose l_\alpha -1}\right) 3^{2d_{AB}-2+k_1}2^{2d_{AB}-2} e^{\Omega_\varepsilon(c_{AC}+c_{BC}+c_{ABC}-k_2)}\cdot \\
&\qquad \cdot \frac{n!(d_{AB}-k_1-k_4)}{(n+d_{AB}-2)!}\frac{(n-c_{AC}-c_{BC}-c_{ABC}-k_2)!}{(n-1-c_{AC}-c_{BC}-c_{ABC})!}
\end{align*}
This can be simplified using 
$${c_{AB}-1\choose l_{AB}-1} \leq \frac{n^{l_{AB}-1}}{(l_{AB}-1)!},$$
$$\frac{n! (d_{AB}-k_1-k_4)!}{(n+d_{AB}-2)!}\leq \frac{d_{AB}^{d_{AB}-k_1-k_4}}{n^{d_{AB}-2}} \leq \varepsilon^{d_{AB}-k_1-k_4} n^{2-k_1-k_4},$$
and equation \eqref{eq:factorialAcupBcapC} to
\begin{align*}
t_n(\tilde{G}, \bar{c}) &\leq \frac{n^{l_{AB}-1}}{(l_{AB}-1)!} \left( \prod_{\alpha\neq AB} {c_\alpha-1 \choose l_\alpha-1} e^{-\Omega_\varepsilon(c_\alpha)}\right) (36\varepsilon)^{d_{AB}} e^{O_\varepsilon(k_1+k_2+k_4)} n^{2-k_1-k_4} n^{1-k_2} \\
&= n^2 \left( \prod_{\alpha\neq AB} {c_\alpha-1 \choose l_\alpha-1}e^{-\Omega_\varepsilon(c_\alpha)}\right) \frac{ (36\varepsilon)^{d_{AB}}  e^{O_\varepsilon(k_1+k_2+k_4)}}{(l_{AB}-1)! n^{k_1+k_2+k_4-l_{AB}}}.
\end{align*}
Note that the edge-ordered triples counted in $S_2$ all satisfy $\abs{A\cap B}>\abs{(A\cup B)\cap C}$, and so there must be at least one edge common to both $A$ and $B$ which is not contained in $C$. Hence for any $\tilde{G}$ considered here, $l_{AB}\geq 1$. Assuming $\varepsilon<\frac{1}{36}$ it follows by summing over $d_{AB}, c_{AC}, c_{BC}, c_{ABC}$ that
\begin{equation}\label{eq:S2}
S_2 = O_\varepsilon(n^2) \sum_{\tilde{G}} \frac{e^{O_\varepsilon(k_1+k_2+k_4)}}{(l_{AB}-1)!n^{k_1+k_2+k_4-l_{AB}}}.
\end{equation}

To bound this sum, we need the following claim.
\begin{claim}\label{claim:S2}
For any reduced common edge graph $\tilde{G}$, $k_1+k_2\geq l_{AB}$.
\end{claim}\begin{proof}
Let $e$ be an edge in $\tilde{G}$ labelled $AB$. Either $e$ is the lowest ordered edge in a common segment of $A$ and $B$, in which case it contributed $1$ to $k_1$, or there is some preceding edge $e'$ in that common segment. As no two adjacent edges in $\tilde{G}$ can be contained in the exact same paths, $e'$ must be labelled $ABC$. Then $e'$ must be last edge in a common segment of $A\cup B$ and $C$, and hence contribute $1$ to $k_2$. This is because a succeeding edge $e''$ in the common segment of $A\cup B$ and $C$ labelled, say, $AC$, cannot be connected to the common end-point of $e$ and $e'$ as this would mean that $A$ contains three edges incident to this vertex, and it cannot be connected to the other end-point of $e'$ as then the path $e''e'e$ would be contained in $A$ which is not monotone.
\end{proof}

By combining Claims \ref{claim:regbound} and \ref{claim:S2} it follows that the number of reduced edge graphs corresponding to given values of $k_1, k_2, k_4$ and $l_{AB}$ is at most $e^{O(k_1+k_2+k_4)}$ if $l_{AB}\leq k_1+k_2$ and $0$ otherwise, which means that
\begin{align*}
&\sum_{\tilde{G}} \frac{e^{O_\varepsilon(k_1+k_2+k_4)}}{(l_{AB}-1)!n^{k_1+k_2+k_4-l_{AB}}} \leq \sum_{l_{AB}=1}^\infty \frac{e^{O_\varepsilon(l_{AB})}}{(l_{AB}-1)!} \sum_{\substack{k_1,k_2\geq 0\\k_1+k_2=l_{AB}}} \frac{e^{O_\varepsilon(k_1+k_2-l_{AB})}}{n^{k_1+k_2-l_{AB}}} \sum_{k_4=0}^\infty \frac{e^{O_\varepsilon(k_4)}}{n^{k_4}}\\
&\qquad \stackrel{s=k_1+k_2-l_{AB}}{=} \sum_{l_{AB}=1}^\infty \frac{e^{O_\varepsilon(l_{AB})}}{(l_{AB}-1)!} \sum_{s=0}^\infty \frac{(l_{AB}+1+s) e^{O_\varepsilon(s)}}{n^{s}} \sum_{k_4=0}^\infty \frac{e^{O_\varepsilon(k_4)}}{n^{k_4}}\\
&\qquad = \sum_{l_{AB}=1}^\infty \frac{O_\varepsilon(1+l_{AB}) e^{O_\varepsilon(l_{AB})}}{(l_{AB}-1)!} = O_\varepsilon(1),
\end{align*}
as desired.

Finally, we wish to bound
\begin{equation*}
S_3 = \sum_{\tilde{G}} \sum_{c_{ABC} > (1-18\varepsilon)n} t_n(\tilde{G}, \bar{c}).
\end{equation*}
Let $d_{ABC}=n-c_{ABC}<18\varepsilon n$. Using the inequalities ${a+b\choose a}\leq 2^{a+b}$ and ${a+b+c\choose a, b, c}\leq 3^{a+b+c}$ to bound the last two binomial coefficients and the trinomial coefficient in Proposition \ref{prop:trip}, we get
\begin{equation}
\begin{split}\label{eq:S3}
&t_n(\tilde{G}, \bar{c}) \leq \frac{n^{l_{ABC}-1}}{(l_{ABC}-1)!} \left(\prod_{\alpha\neq ABC} {c_\alpha - 1\choose l_\alpha - 1} \right) 3^{2d_{ABC}-2-2c_{AB}+k_1} 2^{2d_{ABC}-2-2c_{AB}} \cdot\\
&\quad \cdot 2^{3d_{ABC}-3-c_{AB}-2c_{AC}-2c_{BC}+k_2} \frac{n!(d_{ABC}-c_{AB}-k_1-k_4)! (d_{ABC}-c_{AC}-c_{BC}-k_2)!}{(n+2d_{ABC}-3-c_{AB}-c_{AC}-c_{BC})!}.
\end{split}
\end{equation}
Note that, assuming $\varepsilon<\frac{1}{18}$, any edge-ordered triplet counted in $S_3$ must have one edge common between all paths. Hence we may assume $l_{ABC}$ is always at least $1$ here. As $n+2d_{ABC}-3-c_{AB}-c_{AC}-c_{BC} = \abs{A\cup B\cup C}\geq n-1$ and $d_{ABC}-c_{AB}-k_1-k_4, d_{ABC}-c_{AC}-c_{BC}-k_2 \leq d_{ABC} < 18\varepsilon n$ we get
\begin{align*}
&\frac{n!(d_{ABC}-c_{AB}-k_1-k_4)! (d_{ABC}-c_{AC}-c_{BC}-k_2)!}{(n+2d_{ABC}-3-c_{AB}-c_{AC}-c_{BC})!}\\
&\qquad \leq \frac{(d_{ABC})^{d_{ABC}-c_{AB}-k_1-k_4+d_{ABC}-c_{AC}-c_{BC}-k_2}}{n^{2d_{ABC}-3-c_{AB}-c_{AC}-c_{BC}}}\\
&\qquad \leq (18\varepsilon)^{2d_{ABC}-c_{AB}-c_{AC}-c_{BC}-k_1-k_2-k_4} n^{3-k_1-k_2-k_4}.
\end{align*}

Plugging this into \eqref{eq:S3}, we get
\begin{align*}
&t_n(\tilde{G}, \bar{c}) \leq \frac{n^{l_{ABC}-1}}{(l_{ABC}-1)!} \left(\prod_{\alpha\neq ABC} {c_\alpha-1 \choose l_\alpha-1} 4^{-c_\alpha}\right) 288^{d_{ABC}}\\
&\qquad \cdot (18\varepsilon)^{2d_{ABC}-c_{AB}-c_{AC}-c_{BC}} n^{3-k_1-k_2-k_4} e^{O_\varepsilon(k_1+k_2+k_4)},
\end{align*}
Note that $d_{ABC}-c_{AB}-c_{AC} -1 \geq 0$ as the left side denotes the number of edges of $A\setminus(B\cup C)$, and similarly $d_{ABC}-c_{AB}-c_{BC} -1 \geq 0$ and $d_{ABC}-c_{AC}-c_{BC} -1 \geq 0$. Hence $3d_{ABC}-2c_{AB}-2c_{AC}-2c_{BC} \geq 0$ and thus $2d_{ABC}-c_{AB}-c_{AC}-c_{BC} \geq \frac{d_{ABC}}{2}$. It follows that
\begin{equation*}
\begin{split}
&t_n(\tilde{G}, \bar{c}) \leq n^2 \left(\prod_{\alpha\neq ABC} {c_\alpha-1 \choose l_\alpha-1} 4^{-c_\alpha}\right) (288\sqrt{18 \varepsilon})^{d_{ABC}}\\
&\qquad \cdot \frac{ e^{O_\varepsilon(k_1+k_2+k_4)} }{(l_{ABC}-1)! n^{k_1+k_2+k_4-l_{ABC}}},
\end{split}
\end{equation*}
and thus, assuming $\varepsilon < \frac{1}{288^2\cdot 18}$, we have
\begin{equation*}
S_3 = O_\varepsilon(n^2) \sum_{\tilde{G}} \frac{ e^{O_\varepsilon(k_1+k_2+k_4)} }{(l_{ABC}-1)! n^{k_1+k_2+k_4-l_{ABC}}}.
\end{equation*}

Note that this is identical to the bound we got for $S_2$ except that $l_{AB}$ is replaced by $l_{ABC}$. Hence proceeding as for $S_2$, the following claim implies that $S_3 = O_\varepsilon(n^2)$. This completes the proof that $\E X^3 = O(n^3)$.
\begin{claim}
For any reduced common edge graph $\tilde{G}$, $k_1+k_2\geq l_{ABC}$.
\end{claim}\begin{proof}
Let $e\in\tilde{G}$ be labelled $ABC$. We show that $e$ must either be the first edge in that common segment of $A$ and $B$ or the first edge in that common segment of $A\cup B$ and $C$. Assume that $e'$ is a preceding edge in the common segment of $A$ and $B$, and $e''$ is a preceding edge in the common segment of $A\cup B$ and $C$. Clearly $e'$ is labelled $AB$, and we may assume that $e''$ is labelled $AC$. Then $e'$ and $e''$ cannot connect to the same end-point of $e$ as this would mean that $A$ traverses three edges incident to this vertex, and $e'$ and $e''$ cannot connect to opposite end-points of $e$, as this would mean that $e'ee''$ is a path contained in $A$ which is not monotone.
\end{proof}


It remains to prove the second statement in Proposition \ref{prop:EX3On3}. From the preceding argument we know that the contribution to $\E X^3$ from non-good reduced common edge graphs, and from $\bar{c}$ such that $c_{AB}+c_{AC}+c_{BC} > M$ is at most
\begin{equation}\label{eq:S1prime}
O_\varepsilon(n^2) + 3 \sum_{\tilde{G}} \sum_{\substack{c_{AB}+c_{ABC}\leq (1-\varepsilon)n\\ c_{AC}+c_{BC}+c_{ABC}\leq (1-\varepsilon)n\\c_{AB}+c_{AC}+c_{BC} >M}} \trip_n(\tilde{G}, \bar{c}),
\end{equation}
where the summand is bounded by \eqref{eq:S1summandbound}. For $c_{AB}+c_{AC}+c_{BC}>M$, we can modify the bound in \eqref{eq:S1summandbound} slightly by ``moving half of each factor $e^{-\Omega_\varepsilon(c_\alpha)}$ outside the product over $\alpha$'' to get
\begin{equation*}
\trip_n(\tilde{G}, \bar{c}) \leq n^3 e^{-\Omega_\varepsilon(M/2)} \left( \prod_\alpha{c_\alpha-1\choose l_\alpha-1} e^{-\Omega_\varepsilon(c_\alpha/2)} \right) \frac{ e^{O_\varepsilon(k_1+k_2+k_3+k_4)} }{k_3! n^{k_2+k_4-k_3}}.
\end{equation*}
Note that, up to the implicit constants in the $\Omega_\varepsilon$-notation, this is equivalent to the bound in \eqref{eq:S1summandbound} except that we have an additional factor $e^{-\Omega_\varepsilon(M/2)}$. Hence preceding as before with the bound for $S_1$, it follows that \eqref{eq:S1prime} is bounded by $O_\varepsilon(n^3)e^{-\Omega_\varepsilon(M/2)}$, as desired.
\end{proof}



It remains to investigate the contribution from good reduced common edge graphs where the entries of $\bar{c}$ are of $O(1)$ more carefully. Recall that a good reduced common edge graph $\tilde{G}$, by definition, has no edges common to all three paths. Hence $\trip_n(\tilde{G}, c_{AB}, c_{AC}, c_{BC}, c_{ABC}) = 0$ unless $c_{ABC}=0$.

\begin{proposition}\label{prop:termwiselimit}
Let $\tilde{G}$ be a fixed good reduced common edge graph, and let $c_{AB}, c_{AC}$ and  $c_{BC}$ be fixed non-negative integers. Then,  
\begin{equation*}
\begin{split}
&\lim_{n\rightarrow\infty}  n^{-3}t_n(\tilde{G}, c_{AB}, c_{AC}, c_{BC}, 0)\\
&\qquad =\frac{e^{-6}}{\left( \sum_\alpha k_\alpha\right)!} \prod_\alpha\left( {c_\alpha-r_\alpha-1 \choose k_\alpha-r_\alpha-1} 2^{-c_\alpha+k_\alpha}\right),
\end{split}
\end{equation*}
where the sum and product over $\alpha$ go over $AB, AC,$ and $BC$, $k_\alpha=k_\alpha(\tilde{G})$ denotes the number of edges in $\tilde{G}$ whose label state that the edge is common to the two paths in $\alpha$, and $r_\alpha = r_\alpha(\tilde{G})$ denotes the number of edges whose labels state that the edge is common to the two paths in $\alpha$ and that these paths traverse the edge in opposite directions.
\end{proposition}
\begin{proof}
Given an edge-ordered triple $A, B, C$, we say that an edge $e\in A\cup B \cup C$ \emph{occurs during} an edge set $E \subseteq A \cup B \cup C$ if there exists $e', e''\in E$ such that $e'\preceq e \preceq e''$. Thus, for an edge-ordered triple $A, B, C$, we can assign numbers $m_{AB}, m_{AC}$ and $m_{BC}$ counting the number of unique edges to $C$, $B$ or $A$ for which there exists a common segment of the other two paths that it occurs during.

We now count the number of edge-ordered triples corresponding to fixed $\tilde{G}$, $c_{AB}, c_{AC}, c_{BC}$ and given values of $m_{AB}, m_{AC}$ and $m_{BC}$. Since each common segment of an edge-ordered triple contains at least one edge, we may assume that $c_\alpha \geq k_\alpha=k_\alpha(\tilde{G})$ for $\alpha=AB, AC, BC$.
\\~\\
\textsc{Step 1:} Choose a common edge graph $G$ (up to isomorphism).
\\~\\
Note that, for each $\alpha\in\{AB, AC, BC\}$, any of the $r_\alpha$ edges that the paths traverse in opposite directions cannot be extended to a path with length more than one as this path would have to be increasing in both directions. Hence, for each such $\alpha$, we need to choose the length of each of the $k_\alpha-r_\alpha$ path segments that the paths traverse in the same direction. This is equal to the number of ways to place $c_\alpha - r_\alpha$ indistinguishable balls into $k_\alpha-r_\alpha$ bins such that each bin contains at least one ball. Hence the number of ways to do this is $\prod_\alpha {c_\alpha-r_\alpha-1 \choose k_\alpha-r_\alpha-1 }$. We remark that we have not yet chosen how to embed $G$ into $K_n$, so we have only chosen $G$ up to isomorphism.
\\~\\
\textsc{Step 2:} Choose the edge list of $A$, $B$ and $C$.
\\~\\
Recall that the edge list of $A, B$ and $C$ is a sequence of length $\abs{A\cup B\cup C}$ where the $i$:th element denotes which paths the $i$:th lowest ordered edge among $A\cup B\cup C$ is contained in. Hence, in this case the possible entries are \textsf{A}, \textsf{B}, \textsf{C}, \textsf{AB}, \textsf{AC} and \textsf{BC}. 

Given $G$, we already know the ordering of the $c_{AB}+c_{AC}+c_{BC}$ edges that are common to two paths, so we can start by placing $c_{AB}$ \textsf{AB}:s, $c_{AC}$ \textsf{AC}:s and $c_{BC}$ \textsf{BC}:s in this order. We first extend the list by inserting the position of all $m_{AB}+m_{AC}+m_{BC}$ edges unique to one of the paths that occur during a common segment of the other two. For a given $\alpha$, there are $c_\alpha-k_\alpha$ pairs of consecutive edges $e$ and $e'$ in $\bigcap_{D \in \alpha} D$ between which we could place some or all of these $m_\alpha$ edges. Hence, for each $\alpha$ the number of ways to do this is equal to the number of ways to place $m_\alpha$ indistinguishable balls into $c_\alpha-k_\alpha$ bins, which is ${m_\alpha + c_\alpha-k_\alpha-1 \choose c_\alpha-k_\alpha-1}$.

It remains to choose the placement of the remaining unique edges of each path. This is equivalent to intertwining four strings: one consisting of $n-1-m_{BC}-c_{AB}-c_{AC}$ \textsf{A}:s, one of $n-1-m_{AC}-c_{AB}-c_{BC}$ \textsf{B}:s, one of $n-1-m_{AB}-c_{AC}-c_{BC}$ \textsf{C}:s, and one of $\sum_\alpha k_\alpha$ \textsc{D}:s, where the \textsf{D}:s are placeholders for the common segments. This can be done in $$\frac{ (3n-3-\sum_\alpha (2c_\alpha +m_\alpha- k_\alpha))!}{(\sum_\alpha k_\alpha)! \prod_\alpha (n-1-m_\alpha- \sum_{\alpha'\neq \alpha} c_{\alpha'})!}$$
ways. However, this will overestimate the number of possible edge lists. For instance, it follows from the definition of $\tilde{G}$ being good that we cannot have two \textsf{D}:s next to each other. The exact condition on such an intertwined string to yield a feasible edge list is a bit involved to describe, but as we shall see, it is sufficient that the string contains at least two \textsf{A}:s, two \textsf{B}:s and two \textsf{C}:s between each pair of \textsf{D}:s.


To estimate the proportion of such intertwined strings, consider a random such string, chosen uniformly at random. Observe that the subsequence consisting of all \textsf{A}:s and \textsf{D}:s is uniformly distributed among the possible strings consisting of $n-1-m_{BC}-c_{AB}-c_{AC}$ \textsf{A}:s and $\sum_\alpha k_\alpha$ \textsf{D}:s. Hence, for any two given positions, $i\neq j$, in this list, the probability that there is a \textsf{D} at both position $i$ and $j$ is $$\frac{(\sum_\alpha k_\alpha)\cdot (\sum_\alpha k_\alpha-1)}{(n-1-m_{BC}-c_{AB}-c_{AC}+\sum_\alpha k_\alpha)\cdot(n-1-m_{BC}-c_{AB}-c_{AC}+\sum_\alpha k_\alpha-1)}.$$ In particular, letting $\bar{m}=(m_{AB}, m_{AC}, m_{BC})$, we can bound this probability by $O_{\tilde{G},\bar{c},\bar{m}}\left(\frac{1}{n^2}\right)$. Hence, the probability that this occurs at any two positions $i$ and $i+1$ or $i$ and $i+2$ is $O_{\tilde{G},\bar{c},\bar{m}}\left(\frac{1}{n}\right)$. Arguing in the same way for the subsequences of all \textsf{B}:s and \textsf{D}:s and all \textsf{C}:s and \textsf{D}:s shows that the proportion of ``bad'' strings is $O_{\tilde{G},\bar{c},\bar{m}}\left(\frac{1}{n}\right)$. Note that this estimate might be big if some $c_\alpha$, $k_\alpha$ or in particular $m_\alpha$ is large compared to $n$, but the point is that this estimate will only be important in the setting when $n\rightarrow\infty$ with all other parameters held constant, so this estimate will still suffice.
\\~\\
\textsc{Step 3:} Choose the vertex sequences of $A$, $B$ and $C$.
\\~\\
Given the choice of common edge graph $G$ (up to isomorphism) and edge list of $A$, $B$ and $C$ we will now choose the vertex sequences $(a_1, a_2, \dots, a_n), (b_1, b_2, \dots, b_n)$ and $(c_1, c_2, \dots, c_n)$ denoting in which order each of $A$, $B$ and $C$ will visit the vertices in $K_n$. In doing so, we start by choosing an embedding of $G$ into $K_n$. This can be done in $$\frac{n!}{(n-\sum_\alpha (c_\alpha + k_\alpha))!} =\left(1+o_{\tilde{G},\bar{c}}(1)\right) n^{\sum_\alpha c_\alpha+k_\alpha}$$ ways.

Since we have already chosen the edge list of $A$, $B$ and $C$, we know how many unique edges each path has before its first common segment, between each two common segments, and after its last common segment with one of the other paths. Hence, for any common segment in $G$ between, say, $A$ and $B$, we can identify the corresponding vertex sequences $a_i,  a_{i+1}, \dots, a_{i+l}$ and $b_j, b_{j+1}, \dots, b_{j+l}$ of the two paths. Hence, the embedding of $G$ determines the vertices in $c_{AB}+c_{AC}+k_{AB}+k_{AC}$ positions of the vertex sequence of $A$, and similarly $c_{AB}+c_{BC}+k_{AB}+k_{BC}$ positions in $B$ and $c_{AC}+c_{BC}+k_{AC}+k_{BC}$ ones in $C$. Second, we assign the remaining parts of the vertex sequences of $A$, $B$ and $C$ such that the paths are Hamiltonian, which can be done in $$\prod_\alpha\left( n-\sum_{\alpha'\neq \alpha}(c_{\alpha'}+k_{\alpha'})\right)!=\left(1+o_{\tilde{G},\bar{c}}(1)\right) n!^3 n^{-2\sum_\alpha (c_\alpha + k_\alpha)}$$ ways.

Again, not all of these assignments are allowed as the chosen vertex sequences for $A$, $B$ and $C$ may give rise to more common edges than those in $G$.

\begin{claim}\label{claim:brun}
Assume the edge list of $A$, $B$ and $C$ chosen in \textsc{Step 2} contains at least two \textsf{A}:s, two \textsf{B}:s and two \textsf{C}:s between every pair of common segments. Then, the proportion of assignments of vertex sequences as described in \textsc{Step 3} for which $A$, $B$ and $C$ have no more common edges than those in $G$ is $e^{-6} + o_{\tilde{G}, \bar{c}}(1)$.
\end{claim}
Let us postpone the proof of this claim for now. Fix $\tilde{G}$, $c_{AB}$, $c_{AC}$ and $c_{BC}$. Let $t_n(\tilde{G}, \bar{c}, \bar{m})$ denote the contribution to $\trip_n(\tilde{G}, c_{AB}, c_{AC}, c_{BC}, 0)$ from given values of $\bar{m}=(m_{AB}, m_{AC}, m_{BC})$ as above. By imagining that the assignments in \textsc{Step 2} and \textsc{Step 3} never fail, we get the upper bound
\begin{equation}\label{eq:neverfail}
\begin{split}
&n^{-3}\trip_n(\tilde{G}, \bar{c}, \bar{m})\\
&\leq \frac{\left(1+o_{\tilde{G}, \bar{c}}(1)\right) n^{-3}}{\left( \sum_\alpha k_\alpha\right)!} \left(\prod_\alpha {c_\alpha-r_\alpha-1 \choose k_\alpha-r_\alpha-1} {m_\alpha+c_\alpha-k_\alpha-1 \choose c_\alpha-k_\alpha-1}\right)\cdot \\
&\qquad \cdot \frac{(3n-3-(\sum_\alpha 2c_\alpha+m_\alpha-k_\alpha))!n!^3 n^{-\sum_\alpha (c_\alpha+k_\alpha)}  }{(3n-3- \sum_\alpha c_\alpha)! \prod_\alpha (n-1-m_\alpha-\sum_{\alpha'\neq \alpha} c_{\alpha'})! }.
\end{split}
\end{equation}
Using the bounds \begin{equation*}
\frac{n!}{(n-1-m_\alpha-\sum_{\alpha'\neq\alpha}c_\alpha')!}\leq n^{1+m_\alpha+\sum_{\alpha'\neq\alpha}c_{\alpha'}}
\end{equation*}
 and 
\begin{equation*}
\frac{(3n-3-(\sum_\alpha 2c_\alpha+m_\alpha-k_\alpha))!}{(3n-3- \sum_\alpha c_\alpha)!} \leq \left(\frac{e}{3n(1-o_{\bar{c}}(1))}\right)^{\sum_\alpha c_\alpha+m_\alpha-k_\alpha},
\end{equation*}
where the latter follows from the second part of Lemma \ref{lemma:factorialestimates}, the right-hand side of \eqref{eq:neverfail} simplifies to
\begin{equation}\label{eq:domination}
\frac{\left(1+o_{\tilde{G}, \bar{c}}(1)\right)}{\left( \sum_\alpha k_\alpha\right)!} \left( \prod_\alpha {c_\alpha-r_\alpha-1 \choose k_\alpha-r_\alpha-1} {m_\alpha+c_\alpha-k_\alpha-1 \choose c_\alpha-k_\alpha-1} \left( \frac{e}{3} + o_{\bar{c}}(1)\right)^{c_\alpha+m_\alpha-k_\alpha}\right).
\end{equation}
Note that this expression does not depend on $n$ (besides the error terms). Assuming $n$ is sufficiently large such that, say, $1+o_{\tilde{G}, \bar{c}}(1) \leq 2$ and $\frac{e}{3}+o_{\bar{c}}(1) < 0.95$ this gives a upper bound on $n^{-3} t_n(\tilde{G}, \bar{c}, \bar{m})$ not depending on $n$. Moreover, the sum of \eqref{eq:domination} over all triples of non-negative integers $m_{AB}, m_{AC}, m_{BC}$ is convergent as
$$\sum_{m_\alpha=0}^\infty {m_\alpha+c_\alpha-k_\alpha-1 \choose c_\alpha-k_\alpha-1}0.95^{m_\alpha+c_\alpha-k_\alpha} = 20^{c_\alpha-k_\alpha}<\infty,$$
where we used $\frac{z^l}{(1-z)^l}=\sum_c {c-1\choose l-1} z^c$ with $z=0.95$, $l=c_\alpha-k_\alpha$ and $c=m_\alpha$. Hence, by dominated convergence and the above argument for estimating the proportion of successful assignments in \textsc{Step 2} and \textsc{Step 3}, we see that 
\begin{equation*}
\begin{split}
&\lim_{n\rightarrow\infty}  n^{-3}t_n(\tilde{G}, c_{AB}, c_{AC}, c_{BC}, 0)\\
&= \sum_{\bar{m}} \lim_{n\rightarrow\infty} n^{-3}t_n(\tilde{G}, \bar{c}, \bar{m})\\
&= \frac{e^{-6}}{\left( \sum_\alpha k_\alpha\right)!} \prod_\alpha\left( {c_\alpha-r_\alpha-1 \choose k_\alpha-r_\alpha-1} \sum_{m_\alpha}
 {m_\alpha+c_\alpha-k_\alpha-1 \choose c_\alpha-k_\alpha-1} 3^{-m_\alpha-c_\alpha+k_\alpha} \right)\\
&=\frac{e^{-6}}{\left( \sum_\alpha k_\alpha\right)!} \prod_\alpha\left( {c_\alpha-r_\alpha-1 \choose k_\alpha-r_\alpha-1} 2^{-c_\alpha+k_\alpha}\right),
\end{split}
\end{equation*}
where in the last step we again used $\frac{z^l}{(1-z)^l}=\sum_c {c-1\choose l-1} z^c$ with $z=\frac{1}{3}$, $l=c_\alpha-k_\alpha$ and $c=m_\alpha+c_\alpha-k_\alpha$.
\end{proof}

\begin{proof}[Proof of Claim \ref{claim:brun}.]
Let $I_A$, $I_B$ and $I_C$ denote the indices in the vertex sequences of $A$, $B$ and $C$ respectively that are assigned by choosing the embedding of $G$ into $K_n$. We note that $I_A$ consists of $c_{AB}+c_{AC}+k_{AB}+k_{AC}$ elements contained in $k_{AB}+k_{AC}$ intervals of consecutive indices corresponding to the common segments of $A$ and one of $B$ and $C$. Moreover, as the edge list of $A, B$ and $C$ contains at least two edges unique to $A$ between each pair of common segments between $A$ and some of the other paths, there is at least one index between each pair of such intervals that is not contained in $I_A$. For any edge in $G$ that is labelled to be common to $A$ and, say, $B$, the embedding of $G$ into $K_n$ will assign the corresponding indices $i, i+1\in I_A$ and $j, j+1\in I_B$ vertices $a_i, a_{i+1}, b_j, b_{j+1}$ such that $\{a_i, a_{i+1}\} = \{b_j, b_{j+1}\}$. The analogous statements hold for $B$ and $C$.

Suppose we pick the assignments of the vertex sequences of $A$, $B$ and $C$ according to \textsc{Step 3} in the proof of Proposition \ref{prop:termwiselimit} uniformly at random. We want to show that the probability that all edges common to two of $A$, $B$ and $C$ are of this form is $e^{-6}+o_{\tilde{G},\bar{c}}\left(1\right)$. First, let $i, i+1\in I_A$ be adjacent indices in $A$ corresponding to a common edge between $A$ and $B$. Because of the structure of $G$, the embedding of $G$ will not assign the values $a_i$ or $a_{i+1}$ to any $c_j$, $j\in I_C$. Hence the vertices $a_i$ and $a_{i+1}$ will be placed in $C$ at some indices outside $I_C$. Hence, the probability that these are placed in adjacent positions is at most $$\frac{2}{n-c_{AC}-c_{BC}-k_{AC}-k_{BC}-1} = O_{\tilde{G},\bar{c}}\left(\frac{1}{n}\right).$$ Similarly, if $i\in I_A$ and $i+1\not\in I_A$, then the embedding of $G$ does not assign a value to $a_{i+1}$. Conditioning on the embedding and the assignments of vertex sequences of $B$ and $C$, there are at most $4$ assignments of $a_{i+1}$ such that $\{a_i, a_{i+1}\}$ is contained in one of $B$ and $C$. Hence, the probability that $a_{i+1}$ is chosen in this way is at most $$\frac{4}{n-c_{AB}-c_{AC}-k_{AB}-k_{AC}} = O_{\tilde{G},\bar{c}}\left(\frac{1}{n}\right).$$ By repeating this argument for any permutation of $A, B$ and $C$ and any edge in or adjacent to a common segment, it follows from the union bound that the probability of unwanted overlapping edges in this way is $O_{\tilde{G},\bar{c}}\left(\frac{1}{n}\right)$.

It remains to estimate the probability that $\{a_i, a_{i+1}\}$, $\{b_j, b_{j+1}\}$ and $\{c_k, c_{k+1}\}$ are all distinct for $i, i+1\not\in I_A$, $j, j+1 \not\in I_B$ and $k, k+1\not\in I_C$. Here we make use of Brun's sieve, see e.g. Theorem 8.3.1 in \cite{AS08}: Let $\xi=\xi_n$ be a sequence of non-negative integer-valued random variables. Suppose there exists a constant $\mu$ such that $\E \xi_n \rightarrow \mu$ and moreover $\E {\xi_n \choose r}\rightarrow \frac{\mu^r}{r!}$ for all positive integers $r$ as $n\rightarrow\infty$. Then $\Pr\left(\xi_n=r\right) \rightarrow \frac{\mu^r}{r!}e^{-\mu}$ for any $r=0, 1, \dots$.



Condition on the embedding of $G$ and the vertex sequence of $A$. Let $E$ be the set of all arcs $(a_i, a_{i+1})$ and $(a_{i+1}, a_i)$ for $i, i+1\not\in I_A$, and for any $e=(v, w) \in E$, let $F_e$ be the event that $b_j=v$ and $b_{j+1}=w$ for some $j, j+1\not\in I_B$, that is, $F_e$ is the event that $v$ is immediately followed by $w$ in the vertex sequence of $B$ where the corresponding indices are not contained in $I_B$. Let $\xi = \sum_{e\in E} \mathbbm{1}_{F_e}$. Then $\E {\xi \choose r} = \sum_{\abs{\mathcal{E}}=r} \Pr\left(\bigcap_{F\in\mathcal{E}} F\right)$ where the sum goes over all unordered $r$-tuples of events $F_e$ as above.

We say that an $r$-tuple $\mathcal{E}$ is \emph{compatible} if the corresponding arcs form directed paths in $K_n$ (as opposed to cycles, or some vertex having in- or out-degree more than $1$) and no arc has an end-point on a vertex that the embedding of $G$ has assigned to a common segment of $B$ and $C$. Otherwise, $\mathcal{E}$ is \emph{incompatible}. It is clear that if $\mathcal{E}$ is incompatible, then $\Pr\left(\bigcap_{F\in\mathcal{E}} F\vert\text{assignment of $G$ and $A$}\right)=0$. For a given compatible $r$-tuple $\mathcal{E}$, we may interpret the event $\bigcap_{F\in\mathcal{E}} F$ as that $B$ contains some number $p$ vertex-disjoint directed paths $$(v_0^1, v_1^1, \dots v_{r_1}^1), (v_0^2, v_1^2, \dots v_{r_2}^2), \dots, (v_0^p, v_1^p, \dots v_{r_p}^p)$$ with $\sum_{i=1}^p r_i=r$ where the indices of these paths in $B$ are contained in the complement of $I_B$.

For fixed $c_{AB}, c_{AC}$ and $c_{BC}$ it is straightforward to see that the number of compatible $r$-tuples is $\left(1+o_{\tilde{G},\bar{c},r}(1)\right) \frac{(2n)^r}{r!}$: For the upper bound, the total number of $r$-tuples is $${2\left(n-1-O(c_{AB}+c_{AC}+k_{AB}+k_{AC})\right) \choose r},$$ and for the lower bound, we note that a sufficient condition for an $r$-tuple to be compatible is that no pair of corresponding arcs are adjacent and no arc has an end-point on a common segment of $B$ and $C$, which is true for all but a $o_{\tilde{G},\bar{c},r}(1)$-fraction of all $r$-tuples. Moreover, for any compatible $\mathcal{E}$, $$\Pr\left(\bigcap_{F\in\mathcal{E}} F\middle\vert \text{assignment of $G$ and $A$}\right)=\left(1+o_{\tilde{G},\bar{c},r}(1)\right) n^{-r}$$ as there are $\left(1-o_{\tilde{G},\bar{c},r}(1)\right)n^p$ ways to choose the positions of $p$ intervals of lengths $r_1, r_2, \dots r_p$ in the complement of $I_B$ (morally, $\abs{I_B}$ and $p \leq r$ are small compared to $n$ so most choices of intervals will be disjoint both $I_B$ and each other), and the probability that these intervals are assigned the sequences as above is $$\frac{(n-c_{AB}-c_{BC}-k_{AB}-k_{BC}-r-p)!}{(n-c_{AB}-c_{BC}-k_{AB}-k_{BC})!}=\left(1-o_{\tilde{G},\bar{c},r}(1)\right)n^{-p-r}.$$
Hence, by Brun's sieve, the probability that $\xi=0$, that is, that there are no ``extra common edges'' between $A$ and $B$, is $e^{-2}+o_{\tilde{G},\bar{c}}(1)$.

Conditioning on the embedding of $G$, and the vertex sequences of $A$ and $B$ we can repeat this argument for overlapping edges between $A\cup B$ and $C$. Here we define $E$ as the sets of arcs $(a_i, a_{i+1}), (a_{i+1}, a_i), (b_j, b_{j+1}), (b_{j+1}, b_j)$ for $i, i+1\not\in I_A$ and $j, j+1 \not\in I_B$, with $F_e$ for each $e\in E$ defined as before. For an $r$-tuple $\mathcal{E}$ of such events, we need to add the condition that no arc has an end-point on a vertex that the embedding of $G$ assigns to a common segment of $A$ and $C$ in order for $\mathcal{E}$ to be compatible. There are now $\left(1+o_{\tilde{G},\bar{c},r}(1)\right) \frac{ (4n)^4}{r!}$ possible $r$-tuples $\mathcal{E}$ of events, and, in order for an $r$-tuple to be compatible, it is still sufficient that it avoids pairs of adjacent arcs and arcs with an end-point on a common segment of $A$ and $C$ or $B$ and $C$, which holds true for all but a $o_{\tilde{G},\bar{c},r}(1)$-fraction of all $r$-tuples. We again get that a compatible $r$-tuple of events correspond to that certain sequences of vertices should appear consecutively in $C$, for which we get the same estimate for $\Pr(\bigcap_{F\in\mathcal{E}}F\vert\text{assignment of $G$, $A$ and $B$})$ as before. Again by Brun's sieve, we get that the probability that $\xi=0$ is $e^{-4}+o_{\tilde{G},\bar{c}}(1)$, as desired.
\end{proof}

\begin{proof}[Proof of Proposition \ref{prop:thirdmomentestimates}] By Proposition \ref{prop:EX3On3} we have for any $M>0$ that $n^{-3}\E X^3$ lies beween $\sum_{\tilde{G}\text{ good}}\sum_{c_{AB}+c_{AC}+c_{BC} \leq M} n^{-3} t_n(\tilde{G}, c_{AB}, c_{AC}, c_{BC}, 0)$ and this sum plus $O\left(\frac{1}{n}+e^{-\Omega(M)}\right)$. Taking the limit as $n\rightarrow\infty$ for a fixed $M$ and noting that only a finite number of $\tilde{G}$ contribute to the sum (in order for $\tilde{G}$ to contribute to this sum, it cannot have more than $M$ edges), meaning that we can move the limit inside the sums, gives us
\begin{align*}
&\sum_{\tilde{G}\text{ good}}\sum_{c_{AB}+c_{AC}+c_{BC} \leq M} \lim_{n\rightarrow\infty}n^{-3} t_n(\tilde{G}, c_{AB}, c_{AC}, c_{BC}, 0)\\
&\qquad\leq \liminf_{n\rightarrow\infty}n^{-3} \E X^3 \leq \limsup_{n\rightarrow\infty}n^{-3} \E X^3\\
&\qquad\leq \sum_{\tilde{G}\text{ good}}\sum_{c_{AB}+c_{AC}+c_{BC} \leq M} \lim_{n\rightarrow\infty} n^{-3} t_n(\tilde{G}, c_{AB}, c_{AC}, c_{BC}, 0) + O\left(e^{-\Omega(M)}\right).
\end{align*}
Hence, letting $M\rightarrow\infty$, we get
\begin{equation*}
\lim_{n\rightarrow\infty} n^{-3} \E X^3 = \sum_{\tilde{G}\text{ good}}\sum_{\bar{c}} \lim_{n\rightarrow\infty} n^{-3} t_n(\tilde{G}, \bar{c}).
\end{equation*}
Applying Proposition \ref{prop:termwiselimit} it follows that
\begin{equation*}
\begin{split}
&\lim_{n\rightarrow\infty} n^{-3}\E X^3 = \sum_{\tilde{G}\text{ good}} \frac{e^{-6}}{\left(\sum_\alpha k_\alpha(\tilde{G})\right)!}\cdot\\
&\qquad \cdot \prod_\alpha\left( \sum_{c_\alpha=0}^\infty {c_\alpha-r_\alpha(\tilde{G})-1 \choose k_\alpha(\tilde{G})-r_\alpha(\tilde{G})-1} 2^{-c_\alpha+r_\alpha(\tilde{G})} \cdot 2^{k_\alpha(\tilde{G})-r_\alpha(\tilde{G})}\right)\\
&\qquad =\sum_{\tilde{G}\text{ good}} \frac{e^{-6} }{\left(\sum_\alpha k_\alpha(\tilde{G})\right)!}\prod_\alpha 2^{k_\alpha(\tilde{G})-r_\alpha(\tilde{G})}.
\end{split}
\end{equation*}
Yet again we used $\frac{z^l}{(1-z)^l}=\sum_c {c-1\choose l-1} z^c$, here with $z=\frac{1}{2}$, $l=k_\alpha\left(\tilde{G}-r_\alpha(\tilde{G})\right)$, $c=c_\alpha-r_\alpha(\tilde{G})$. Note that for given $k_{AB}, k_{AC}, k_{BC}$ and $r_{AB}, r_{AC}, r_{BC}$ the number of corresponding good reduced common edge graphs is
$ { k_{AB}+k_{AC}+k_{BC} \choose k_{AB}, k_{AC}, k_{BC}}\prod_\alpha {k_\alpha \choose r_\alpha}.
$
Hence
\begin{equation*}
\begin{split}
\lim_{n\rightarrow\infty} n^{-3} \E X^3 &= e^{-6} \prod_{\alpha} \sum_{k_\alpha=0}^\infty \frac{ 2^{k_\alpha}}{k_\alpha!} \sum_{r_\alpha=0}^{k_\alpha} {k_\alpha \choose r_\alpha} 2^{-r_\alpha}\\
&= e^{-6} \prod_{\alpha} \sum_{k_\alpha=0}^\infty \frac{ 3^{k_\alpha}}{k_\alpha!}\\
&= e^{3}.
\end{split}
\end{equation*}
The estimates for \eqref{eq:EX3CdisjointA} and \eqref{eq:EX3BCdisjointA} are done analogously, but only count the contributions from terms where $k_{AC}=0$ and $k_{AC}=k_{BC}=0$ respectively.
\end{proof}

\section*{Acknowledgements}
Most of this work was done at the Department of Mathematical Sciences at Chalmers University of Technology and University of Gothenburg as part of the PhD studies of the author. This was supported by a grant from the Swedish Research Council. The author thanks Klas Markström for suggesting the problem, and his supervisor, Peter Hegarty, as well as the anonymous referees for their thorough reading of the manuscript and many valuable comments and suggestions.

\end{document}